\documentclass[11pt,oneside]{article}
\usepackage[paper=a4,DIV=11]{typearea}
\usepackage{amssymb,amsmath,amsfonts,amsthm}
\usepackage[english]{babel}
\usepackage[T1]{fontenc}
\usepackage[utf8]{inputenc}
\usepackage{enumerate}
\usepackage{tikz,float}
\usetikzlibrary{patterns}
\usepackage[
  backgroundcolor=yellow,
  linecolor=black,
  bordercolor=red,
  textsize=tiny,
  shadow,
  ]{todonotes}
  \newcounter{todocounter}

\usepackage[blocks]{authblk}

\renewcommand{\epsilon}{\varepsilon}
\newcommand{\euler}{\mathrm{e}}
\newcommand{\sfuc}{\mathrm{sfuc}}

\newcommand{\RR}{\mathbb{R}}
\newcommand{\CC}{\mathbb{C}}
\newcommand{\NN}{\mathbb{N}}	
\newcommand{\ZZ}{\mathbb{Z}}
\newcommand{\TT}{\mathbb{T}}

\newcommand{\EE}{\mathbb{E}}
\newcommand{\cL}{\mathcal{L}}
\newcommand{\cF}{\mathcal{F}}

\renewcommand{\Re}{\operatorname{Re}}
\newcommand{\supp}{\operatorname{supp}}
\newcommand{\x}{z}
\newcommand{\ran}{\operatorname{Ran}}
\DeclareMathOperator{\funs}{s}
\newtheorem{theorem}{Theorem}[section]
\newtheorem{lemma}[theorem]{Lemma}
\newtheorem{proposition}[theorem]{Proposition}
\newtheorem{corollary}[theorem]{Corollary}
\theoremstyle{definition}
\newtheorem{definition}[theorem]{Definition}
\newtheorem{example}[theorem]{Example}
\theoremstyle{remark}
\newtheorem{remark}[theorem]{Remark}
\usepackage{microtype}
\usepackage{ellipsis}
\begin{document}
%
%
%
%
\title{Harmonic Analysis and Random Schr\"odinger Operators}
\author{Matthias T\"aufer}
\author{Martin Tautenhahn}
\author{Ivan Veseli\'c}
\affil{Technische Universit\"at Chemnitz \\ Fakult\"at f\"ur Mathematik \\ 09126 Chemnitz, Germany}
\date{\vspace{-2em}}
\maketitle
\tableofcontents
%
%
%
%
%
%
%
%
\section{Introduction}

This survey is based on a series of lectures given during the \emph{School on Random Schr\"odinger Operators} and the \emph{International Conference on Spectral Theory and Mathematical Physics} at the Pontificia Universidad Catolica de Chile, held in Santiago in November 2014.
As the title suggests, the presented material has two foci:
Harmonic analysis, more precisely, unique continuation properties of several natural function classes and Schr\"odinger operators, more precisely properties of their eigenvalues, eigenfunctions and solutions of associated differential equations.
It mixes topics from (rather) pure to (rather) applied mathematics, as well as classical questions and results dating back a whole century
to very recent and even unpublished ones.
The selection of material covered is based on the selection made for the minicourse, and is certainly a personal choice corresponding to the research interests of the authors.
\par
Emphasis is layed not so much on proofs, but rather on concepts, questions, results, examples and applications. In several cases, however, we do supply proofs of special cases or sketches of proofs, and use them to illustrate the underlying concepts.
As the minicourse \emph{Harmonic Analysis and Random Schr\"odinger Operators} itself, we designed the text to be accessible to advanced graduate students which have already aquired some experience with partial differential equations.
On the other hand, even experts in the field will find new results, mostly toward the end of the text.
\par
The line of thought starts with discussing unique continuation properties of holomorphic and harmonic functions. Already here we illustrate different notions of unique continuation.
Hereafter, elliptic partial differential equations are introduced and unique continuation properties of their solutions are discussed.
Then we shift our attention to domains and differential equations with an inherent multiscale structure.
The question here is, whether appropriately collected local data of a function give good estimates to global properties of the function.
In the framework of harmonic analysis the Whittaker--Nyquist--Kotelnikov--Shannon Sampling and the Logvinenko-Sereda Theorem are examples of such results.
From here it is natural to pursue the question whether similar and related results can be expected for (classes of) solutions of differential equations.
This leads us to quantitative unique countinuation bounds which are obtained by the use of Carleman estimates.
In the context of random Schr\"odinger operators they have risen to some prominence recently since they facilitated the resolution of some long standing problems in the field.
We present several unique countinuation theorems tailored for this applications.
Finally, after several results on the spectral properties of random  Schr\"odinger operators, an application to control of the heat equation is given.
\subsection{Unique continuation}
Intuitively, a unique continuation property describes the phenomenon that certain global properties of appropriately chosen function classes are uniquely determined by knowledge of the function locally, that is on arbitrarily small balls around a reference point.
The following definition is classic. We denote by $B (x,r) = \{y \in \RR^d \mid \lvert x-y \rvert < r\}$ the open ball with center $x \in \RR^d$ and radius $r > 0$. If $x = 0$ we write $B (r)$ instead of $B (0,r)$.
\begin{definition}  \label{def:ucp}
 Let $\Omega \subset \RR^d$ be open.
 A class of functions $\cF = \cF(\Omega) \subset \{f\colon\Omega \to \CC\mid   f \text{ measurable}\}$ satisfies the
  \begin{itemize}
  \item \emph{(weak) unique continuation property}, if every $f \in \cF$ which vanishes on a nonempty and open subset $W \subset \Omega$ vanishes everywhere. In other words, we have the implication
 \begin{equation}
 \label{eq:weak_ucp}
  \exists \,W \,  \subset \Omega\ \text{nonempty and  open, with } f \equiv 0\ \text{on}\ W \Rightarrow f \equiv 0.
 \end{equation}
\item \emph{strong unique continuation property}, if every $f \in \cF$ which vanishes of every polynomial order at some point $x_0 \in \Omega$ vanishes everywhere. In other words, we have the implication
 \begin{equation}
  \label{eq:strong_ucp}
  \exists \, x_0 \in \Omega\ \text{such that}\ \forall\, N \in \NN : \lim_{\epsilon \to 0} \epsilon^{-N} \int_{B(x_0, \epsilon)} \lvert f \rvert \mathrm{d} x = 0 \Rightarrow f \equiv 0.
 \end{equation}
\end{itemize}
\end{definition}
In the present manuscript we also introduce the following notions which
have been considered previously in the literature and/or are suitable for the discussion which follows.
\begin{remark}
 Let $\Omega \subset \RR^d$ be open.
 A class of functions $\cF \subset \{f\colon\Omega \to \CC\mid   f \text{ measurable}\}$ satisfies the
\begin{itemize}
 \item \emph{semi-strong unique continuation property}, if every $f \in \cF$ which vanishes of some exponential order at some point $x_0 \in \Omega$ vanishes everywhere. In other words, we have the implication
 \begin{equation}
  \label{eq:semi-strong_ucp}
  \exists \, x_0 \in \Omega \text{ and  } a,b > 0 \text{ with } \lim_{\epsilon \to 0} \euler^{a \epsilon^{-b}} \int_{B(x_0, \epsilon)} \lvert f \rvert \mathrm{d} x = 0 \Rightarrow f \equiv 0.
 \end{equation}
\item \emph{very strong unique continuation property of order $N_0 > 0$}, if there is an $\epsilon$-polynomial lower bound of order $N_0$ for the $\cL^1$-norm
of $0 \not \equiv f \in \cF$ on $\epsilon$-balls.
More precisely, if for each $x_0\in \Omega$ and $0 \not \equiv f \in \cF$ there is a constant $C=C(x_0, f)$ and a radius $\epsilon_0=\epsilon_0(x_0, f)\in (0,\infty)$ such that
\begin{equation}
  \label{eq:very_strong_ucp}
C\epsilon^{N_0} \leq  \int_{B(x_0, \epsilon)} \lvert f \rvert \mathrm{d} x \quad \text{ for all } \epsilon\in (0,\epsilon_0).
 \end{equation}
\end{itemize}
The notion ``weak'' to ``very strong'' makes sense since \eqref{eq:very_strong_ucp} $\Rightarrow$
\eqref{eq:strong_ucp} $\Rightarrow$
\eqref{eq:semi-strong_ucp} $\Rightarrow$
\eqref{eq:weak_ucp}.
In fact, the only non-trivial implication is \eqref{eq:very_strong_ucp} $\Rightarrow$ \eqref{eq:strong_ucp}, so let us give a short proof.
\begin{proof}[Proof of (4) $\Rightarrow$ (2)]
 Assume that $f$ satisfies the very strong unique continuation property of order $N_0 \in \NN$ and that there is $x_0 \in \Omega$ such that for all $N \in \NN$ (hence in particular for some $N > N_0$) we have $\lim_{\epsilon \to 0} \epsilon^{-N} \int_{B(x_0, \epsilon)} \lvert f \rvert \mathrm{d}x = 0$. Using $f \not \equiv 0$, we find by the very strong unique continuation property
 $\lim_{\epsilon \to 0} \epsilon^{-N} \int_{B(x_0, \epsilon)} \lvert f \rvert \mathrm{d}x \geq \lim_{\epsilon \to 0} C \epsilon^{N_0 - N} = \infty$ since $N > N_0$, a contradiction.
\end{proof}
It makes sense to consider uniform variants of these properties, for instance uniform w.r.t.~the center of the ball $x_0$ or uniform w.r.t.~the functions in the set $\cF$.
Sometimes such uniformity is easy to achieve, somtimes not.
A nice example where compactness and periodicity are used to enhance a simple unique continuation property to a unique continuation property, uniform over several scales, is given in Section 4 of \cite{CombesHK-03}.
\par
In particular one has the following uniform variants of the \emph{very strong unique continuation property of order $N_0$}:
We say that a class of functions $\cF \subset \{f\colon\Omega \to \CC\mid   f \text{ measurable}\}$ satisfies the
\begin{itemize}
 \item
 very strong unique continuation property of order $N_0$, \emph{uniform in the base point},
 if for every $0\not\equiv f \in \cF$ there is a constant $C=C(f)$ and a radius $\epsilon_0=\epsilon_0( f)\in (0,\infty)$
such that
\begin{equation*}
C\epsilon^{N_0} \leq  \int_{B(x_0, \epsilon)} \lvert f \rvert \mathrm{d} x \quad \text{ for all } x_0 \in \Omega \text{ and } \epsilon\in (0,\epsilon_0).
 \end{equation*}
\end{itemize}
It may well happen that the behaviour of functions in $\cF$ near the boundary of $\Omega$ is less regular
than, say, the $r$-interior $\Omega_r := \{x \in \Omega \mid \operatorname{dist} (x ,\partial \Omega)>r\}$ for some $r>0$.
In this case we would not have the above type of uniformity. We also say that $\cF$ satisfies the
\begin{itemize}
 \item
 very strong unique continuation property of order $N_0$, \emph{uniform in the set $\cF$},
 if for every $x_0 \in \Omega$ there is a radius $\epsilon_0=\epsilon_0(x_0)\in (0,\infty)$ such that
 for all  $0\not\equiv f \in \cF$ there is a constant $C=C(x_0,f)\in (0,\infty)$ with
\begin{equation*}
C\epsilon^{N_0} \leq  \int_{B(x_0, \epsilon)} \lvert f \rvert \mathrm{d} x \quad
\text{ for all }
\epsilon\in (0,\epsilon_0).
 \end{equation*}
\item
very strong unique continuation property of order $N_0$, \emph{uniform in the base point and in the set $\cF$},
 if there is a radius $\epsilon_0\in (0,\infty)$ such that for every $0\not\equiv f \in \cF$
 there is a constant a constant $C=C(f)\in (0,\infty)$ with
\begin{equation}
\label{eq:uniform-very-strong}
C\epsilon^{N_0} \leq  \int_{B(x_0, \epsilon)} \lvert f \rvert \mathrm{d} x \quad
\text{ for all } x_0 \in \Omega \text{ and } \epsilon\in (0,\epsilon_0).
 \end{equation}
\end{itemize}
One might wonder whether the constant $C=C(f)$ could be chosen uniform in $\cF$, as well.
This cannot be expected if $\cF$ is closed under scalar multiplication, as it is the case for vector spaces, since then for sufficiently small $\lambda > 0$
\[
\int_{B(x_0, \epsilon)} \lvert \lambda f \rvert \mathrm{d} x
=
\lambda \int_{B(x_0, \epsilon)} \lvert f \rvert \mathrm{d} x
< C \epsilon^{N_0} .
\]
Thus we see that it will be natural to complement the requirement $f \in \cF$ with some kind of normalization,
e.g.\ $\int \lvert f \rvert^p=1$. Alternatively, the normalization can be already taken care of in the function class $\cF$.
Then we would be dealing, e.g., with the unit sphere in a normed linear space.
In this situation one can obviously drop the condition $f \not \equiv 0$ which appeared several times above.
\end{remark}
\begin{remark}
\label{r:normalization}
Another way to allow $\cF$ to be a vector space would be to multiply the left hand side of \eqref{eq:uniform-very-strong} with the norm of $f$.
Later, in Section \ref{s:multiscale}, we will do this, but in an $\cL^2$-setting.
This means that we will study inequalities of the form
\begin{equation}
 \label{eq:uniform-very-strong-connection-ucp}
 C \epsilon^{N_0} \int_\Omega \lvert f \rvert^2 \mathrm{d}x
 \leq
 \int_{B(x_0, \epsilon)} \lvert f \rvert^2 \mathrm{d} x \quad
\text{ for all } x_0 \in \Omega \text{ and } \epsilon\in (0,\epsilon_0)
\end{equation}
and similar expressions where $B(x_0, \epsilon)$ has been replaced by a more general set, e.g.\ a disjoint union of $\epsilon$-balls. We will call estimates as in \eqref{eq:uniform-very-strong-connection-ucp} quantitative unique continuation estimates.
\end{remark}
\subsection{Harmonic and holomorphic functions}
\begin{example}[Polynomials of degree one on $\RR$]
 Let $\cF = \mathcal{P}_1(\RR)$ be the space of affine polynomials on $\RR$ with degree at most one, that is $\Delta f = 0$, where $\Delta$ denotes the Laplace operator or the second derivative.
Every $f \in \mathcal{P}_1(\RR)$ can be written as $f(x) = ax + b$ where $a,b \in \RR$. Now there are three possibilities:
 \begin{itemize}
  \item If $a \neq 0$, there is exactly one root and $f$ vanishes on no ball $B(x_0, \epsilon)$
  \item If $a = 0$, $b \neq 0$, then $f$ never vanishes.
  \item If $a = 0$, $b = 0$, then $f \equiv 0$ on $\RR^d$.
 \end{itemize}
 Thus, $\cF$ satisfies the weak unique continuation property as well as the semi-strong and the strong unique continuation property.
Moreover, $\cF$ satisfies the very strong unique continuation property of order $2$ since a non-zero function $f \in \mathcal{P}_1(\RR)$ can vanish at most of order $1$.
 \end{example}
\begin{example}[Harmonic and holomorphic functions]
 One can generalize this to higher dimensions and an open connected $\Omega \subset \RR^d$. The space of harmonic functions on $\Omega$ is $\{ f\in C^2(\Omega) \mid \Delta f \equiv 0 \}$.
 It is known, see for example \cite{Rudin-70}, that such functions are real analytic and thus the space of harmonic functions satisfies the weak, the semi-strong and the strong unique continuation property.
 The same holds for holomorphic functions $\CC \to \CC$.
\end{example}
By definition, the various unique continuation properties above concern local behaviour of a function at a point.
Considering certain natural classes of functions one observes that there is a  connection to global properties, for instance the growth behaviour at infinity.
\begin{example}[A counterexample] \label{ex:powers}
For $k \in \NN$ let $f_k : \CC \to \CC$, $z \mapsto z^k$. Since $f_k$ is holomorphic, it is analytic and hence satisfies the weak, the semi-strong and the strong unique continuation property.
 For large $k$, however, $f_k$ vanishes arbitrarily fast at $z_0 = 0$.
 Thus, for any $N_0$ the space $\{ f : \CC \to \CC \mid f\ \text{holomorphic} \}$ fails to satisfy the very strong unique continuation property \eqref{eq:very_strong_ucp} of order $N_0$.
 Furthermore, all $f_k$ are uniformly bounded on $B(1)$ by $1$. Thus, a local bound is not sufficient for very strong unique continuation.
 However, we observe that for large $k$, $f_k$ grows fast at infinity.
 \end{example}
One might hope that nonzero holomorphic functions cannot vanish faster at $0$ than they grow at infinity.
This observation is made more precise in the following theorem and its corollary.
It is known as Hadamard's three circle theorem and can for instance be found in \cite{Littlewood-1912}, where it is stated as an already known result.
 \begin{theorem}[Hadamard's three circle theorem] \label{t:hadamard3circle}
  Let $r_1 < r_2 < r_3$, $f$ be a holomorphic function on the annulus $r_1 \leq \lvert z \rvert \leq r_3$ and $M_f(r_i) := \max_{ \lvert z \rvert = r_i} \lvert f(z) \rvert$.
 Then
 \begin{equation}
 \label{eq:hadamard}
  \log \left( \frac{r_3}{r_1} \right) \log M_f(r_2)
  \leq
  \log \left( \frac{r_3}{r_2} \right) \log M_f(r_1)
  +
  \log \left( \frac{r_2}{r_1} \right) \log M_f(r_3) .
\end{equation}
 \end{theorem}
 If we choose $\epsilon = r_3 / r_2 = r_2 / r_1$, then \eqref{eq:hadamard} becomes
 \[
   2\log M_f(r_2) \leq \log M_f(\epsilon r_2) + \log M_f(r_2 / \epsilon) .
 \]
Thus the theorem is a statement about convexity of the map $\log (r) \mapsto \log M_f(r)$.
 \begin{corollary}\label{cor:growth-vs-vanishing}
  Let $f : \CC \to \CC$ be holomorphic. Assume that $f$ grows slower at $\infty$ than it vanishes at $0$, i.e. we have
  \begin{align*}
   \liminf_{\epsilon \to 0} M_f(  \epsilon) \cdot M_f(1/ \epsilon) = 0 .
  \end{align*}
  Then $f \equiv 0$.
 \end{corollary}
\begin{proof}
Let $z_0 \in \CC$ with $\lvert z_0 \rvert = 1$.
We apply Hadamard's three circle theorem with $r_1 = \epsilon$, $r_2 = 1$ and $r_3 = 1 / \epsilon$ and obtain for all $\epsilon > 0$
\[
  2  \log M_f(1)
  \leq
   \log M_f(\epsilon )
  +
   \log M_f(1 / \epsilon)
\]
and thus
\[
 \lvert f(z_0)\rvert^2 \leq M_f(1)^2 \leq  M_f(\epsilon) \cdot M_f(1 / \epsilon).
\]
Letting $\epsilon$ tend to $0$, we find by our assumption that $f \equiv 0$ on $\{z \in \CC \mid \lvert z \rvert = 1\}$. Since $f$ is holomorphic, $f \equiv 0$ on $\{z \in \CC \mid \lvert z \rvert \leq 1\}$ by the maximum principle. By analyticity we obtain $f \equiv 0$.
\end{proof}
Instead of holomorphic functions $f_k: \CC \to \CC, z \mapsto z^k$, we could also have considered the harmonic functions $F_k: \RR^2 \to \RR^2, (x,y) \mapsto \mathrm{Re} ( x + i y)^k$ where we use the identification $\CC \cong \RR^2$. Since there is a natural connection between holomorphic and harmonic functions, namly the real and imaginary part of every holomorphic function are harmonic, we would have found similar relations between vanishing at $0$ and growth at $\infty$ for harmonic functions on $\RR^2$.
\par
Another example concerns the spherical harmonics on the sphere, cf.~\cite{ColindeVerdiere-85}.
\begin{example}[Spherical harmonics] \label{ex:spherical-harmonics}
 Let $\mathbb{S}^2 := \{ x \in \RR^3 \mid \lvert x \rvert = 1 \}$ be the $2$-sphere. There is a special orthonormal base of $\cL^2(\mathbb{S}^2)$, called the spherical harmonics $\{ Y_{l,m} \mid l \in \NN, -l \leq m \leq l \}$ such that
 \begin{equation*}
  \begin{cases}
  - \Delta Y_{l,m} &= l(l+1) Y_{l,m} \quad \text{and} \\
  \frac{\partial}{\partial \phi} Y_{l,m} &= i m Y_{l,m},
  \end{cases}
 \end{equation*}
where $\partial / \partial \phi$   denotes the derivative with respect to the $\phi$ coordinate in spherical coordinates.
\par
We study the sequence $Y_{l,l}$, $l \in \NN$. In spherical coordinates they are of the form
\[
Y_{l,l} = c_l  \cos(\theta)^l \exp (\mathrm{i} l \phi ), \quad \theta \in [-\pi / 2 , \pi / 2], \quad \phi \in [0,2\pi),
\]
where $c_l > 0$ is a normalization factor.
\par
Letting $E_r $ be a tubular neighborhood around the equator, that is $E_r  := \{ (\sigma,\theta) \in \mathbb{S}^2 \mid \lvert \theta \rvert < r  \}$, then the mass of $Y_{l,l}$ concentrates exponentially around the equator if $l$ tends to $\infty$, i.e.\ there is $C = C(r ) > 0$ such that
\begin{equation}
 \label{eq:concentration_at_equator}
 \lim_{l \to \infty} \euler^{C(r ) l} \int_{\mathbb{S}^2 \setminus E_r } Y_{l,l} = 0.
\end{equation}
The interesting points to consider are at the poles and we will consider the order of vanishing of the eigenfunctions
at these points.
\par
If we consider the class of functions $\cF = \{Y_{l,l} \mid l \in \{1, \ldots,l_{\max}\}\}$  for some $l_{\max}\in \NN$, then the uniform very strong unique continuation principle as in \eqref{eq:uniform-very-strong} is satisfied, as the following calculation shows.
\par
Since the only zero of $Y_{l,l}$ is at the pole, we have for all $l = 1, ..., l_{\max}$, all $x_0 \in \mathbb{S}^2$ and all $\epsilon < \pi /2$
\[
 \int_{B(x_0, \epsilon)} \lvert Y_{l,l} \rvert \mathrm{d}x \geq \int_{B(p, \epsilon)} \lvert Y_{l,l} \rvert \mathrm{d}A  ,
\]
where $p$ is a pole (by symmetry, we can assume that $p$ is the north pole).
Note that balls on the sphere are defined with respect to the geodesic distance.
Now,
\begin{align*}
 \int_{B(p, \epsilon)} \lvert Y_{l,l} \rvert \mathrm{d}x
&  = \int_0^{2 \pi} \mathrm{d} \phi \int_{\pi/2 - \epsilon}^{\pi/2} \mathrm{d} \theta c_l \cos (\theta)^l \sin(\theta)   \\
&  = 2 \pi c_l \int_{\pi/2 - \epsilon}^{\pi/2} \cos (\theta)^l \sin(\theta) \\
&  = \frac{2 \pi c_l}{l + 1} \cos (\pi/2 - \epsilon)^{l+1} = \frac{2 \pi c_l}{l + 1} \sin(\epsilon)^{l+1}.
\end{align*}
The function $\epsilon \mapsto \sin(\epsilon)^{l+1}$ vanishes of order $l+1$ at $0$. Thus for every $Y_{l,l}$, there is an $l+1$-polynomial lower bound, uniform on $\mathbb{S}^2$, i.e. there is $C = C(Y_{l,l}) > 0$ such that
\[
 \int_{B(x_0, \epsilon)} \lvert f \rvert \geq C(f) \epsilon^{l+1}\ \text{for all}\ x_0 \in \mathbb{S}^2,\ \epsilon < \pi/2.
\]
Since in this case, $\cF$ is a finite set, we can choose $C = \min_{l = 1}^{l_{\max}} C(Y_{l,l})$
and find the uniform very strong unique continuation principle of order $N_0 = l+1$ as in \eqref{eq:uniform-very-strong}.
\par
On the other hand the set $\{Y_{l,l} \mid l \in \NN\}$ does not satisfy the uniform very strong unique continuation principle.
In fact, given $N_0 > 0$, we see by the above calculation that for $l_0 = \lceil N_0 \rceil \in \NN$,
the function $Y_{l_0, l_0}$ vanishes of order $l_0 + 1 > N_0$ at the poles, thus
\eqref{eq:uniform-very-strong} cannot hold.
\end{example}
The limit in \eqref{eq:concentration_at_equator} tells us that for high energies (high eigenvalues)
the eigenfunctions are more and more unevenly distributed on the sphere.
Of course, the choice of eigenbasis for the Laplace operator on the sphere and the 'diagonal' subsequence plays a crucial role here.
Since the eigenvalues of the Laplacian on the sphere are highly degenerate one has a lot of freedom when choosing an orthonormal basis of eigenfunctions.
With an appropriate choice of basis and enumeration, it may be well possible that eigenfunctions for high eigenvalues do obey an equidistribution or quantum ergodicity property on the sphere excluding a behaviour like \eqref{eq:concentration_at_equator}.
In fact, this has been established to hold almost surely for a random choice of eigenbasis by Zelditch.
\par
Note that the $l$-th eigenvalue level of $- \Delta$ on $\mathbb{S}^2$ is $(2l - 1)$-fold degenerate.
Thus, the set of all possible choices of orthonormal bases of $\cL^2(\mathbb{S}^2)$, consisting of eigenfunctions of $\Delta$, can be identified with the product $U(1) \times U(3) \times U(5) \times ...$ where $U(n)$ denotes the unitary group on $\CC^n$.
This product naturally carries the structure of a probability measure, the Haar measure $\mu_{\mathrm{Haar}}$, see \cite{Zelditch-92} for details. The following result can be found in \cite{Zelditch-92}. We formulate a simplified version for multiplication operators.
\begin{theorem}[Almost sure quantum ergodicity on the sphere]
Let $f \in C^\infty(\mathbb{S}^2)$. For $\mu_{\mathrm{Haar}}$-almost every orthonormal basis $(\phi_j)_{j \in \NN}$ of $-\Delta$-eigenfunctions on $\cL^2(\mathbb{S})$, that is $- \Delta \phi_j = E_j \phi_j$, we have
\begin{align*}
\lim_{E \to \infty} \frac{1}{N(E)}\sum_{j \in \NN; E_j \leq E} \lvert \langle \phi_j, f\phi_j\rangle -\bar f  \rvert^2 = 0
\end{align*}
where $\bar f = \mathrm{Vol}(\mathbb{S}^2)^{-1}\int_{\mathbb{S}^2} f(x) \mathrm{d}x$ and $N(E)$ ist the number of eigenvalues not exceeding the energy $E$.
\end{theorem}
It is much harder to find a specific, deterministic eigenbasis for the Laplacian on the sphere with the quantum ergodicity property. A corresponding conjecure and first steps of its proof can be found in \cite{BoechererSS-03}.
Using a different method a deterministic eigenbasis with the quantum ergodicity property was found very recently in \cite{BrooksML-15}.
\par
These examples for the behaviour of Laplace eigenfunctions on the sphere were remerkable because they clarified that one has to be careful with analogies between ergodicity or integrability properties of a classical system and its quantum analogue.
%
%
%
%
%
%
%
%
%
%
%
\section{Vanishing speed for solutions of elliptic PDE}
One can generalize the study of unique continuation properties to solutions of a large class of partial differential operators.
A milestone result is \cite{Carleman-39}.
There unique continuation properties for solutions of a system of first order differential equations with sufficiently regular coefficients on open subsets $\Omega$ of $\RR^2$ are proven. Note that second order partial differential equations can be transformed into a system of first order differential equations, see for instance \cite{Hadamard-03}, page 348.
For this purpose, Carleman introduced a new method which is nowadays called Carleman estimates.
While Carleman's original result applies to the two-dimensional case only,
it has been generalized to arbitrary dimensions in \cite{Mueller-54} and
 by now there are plenty of results concerning Carleman estimates and their applications.
\par
An example of a Carleman estimate, see \cite{KenigRS-86,Kenig-86} and the references therein, is the following:
for all $u \in C_0^\infty(\RR^d)$ and $p, p'$ with $1/p - 1/p' = 2/d$, and all sufficiently large $\lambda > 0$
we have
\begin{equation} \label{eq:CarlemanKRS}
 \lVert \euler^{- \lambda x_d} u \rVert_{\cL^{p'}(\RR^d)} \leq C \lVert e^{- \lambda x_d} \Delta u \rVert_{\cL^p(\RR^d)} ,
\end{equation}
where $x_d$ denotes the $d$-th coordinate of $x$. In fact, Ineq.~\eqref{eq:CarlemanKRS} can be extended
to $\{ u \in \cL^{p'}(\RR^d) \mid \Delta u \in \cL^p(\RR^d)\ \text{and}\ \exists \mu \in \RR:\ \operatorname{supp} u \subset \{ x_d > \mu \} \}$.
\begin{example}[How to conclude UCP from Carleman]
We follow \cite{Kenig-86} and show how the Carleman estimate~\eqref{eq:CarlemanKRS} can be used to obtain a unique continuation property.
Let $V \in \cL^{d/2} (\mathbb{R}^d)$ and, as before,  $1/p - 1/p' = 2/d$. Our goal is to show that if $u \in C_0^\infty (\RR^d)$ satisfies $\lvert \Delta u \rvert \leq \lvert V u \rvert$ and $\operatorname{supp} u \subset \{x_d > 0\}$, then $u \equiv 0$.
\begin{proof}
In a first step, we show that $u$ vanishes on a strip $S_\rho = \{x \in \RR^d \mid x_d \in [0,\rho]\}$, $\rho > 0$.
We choose $\rho > 0$ to be the largest number such that
\[
 C \lVert V \rVert_{\cL^{d/2} (S_\rho + x_d \cdot e_d)} \leq \frac{1}{2}\ \text{for all}\ x_d \in \RR.
\]
where $C$ is the constant from the Carleman estimate~\eqref{eq:CarlemanKRS} and $e_d$ the unit vector in the $d$-th dimension.
Such a $\rho$ exists since $V \in \cL^{d/2}(\RR^d)$.
Now, inequality~\eqref{eq:CarlemanKRS} gives for all $\lambda > 0$
\[
 \bigl\lVert \mathrm{e}^{-\lambda x_d} u \bigr\rVert_{\cL^{p'} (S_\rho)}
 \leq
 C \bigl\lVert \mathrm{e}^{-\lambda x_d} V u \bigr\rVert_{\cL^{p} (S_\rho)}
 +
 C \bigl\lVert \mathrm{e}^{-\lambda x_d} \Delta u \bigr\rVert_{\cL^{p} (\mathbb{R}^d \setminus S_\rho)} .\vspace{1ex}
\]
By H\"older's inequality and our assumption on $\rho$ we obtain
\begin{align*}
 \bigl\lVert \mathrm{e}^{-\lambda x_d} u \bigr\rVert_{\cL^{p'} (S_\rho)}
 &\leq C \lVert V \rVert_{\cL^{d/2} (S_\rho)}  \bigl\lVert \mathrm{e}^{-\lambda x_d}  u \bigr\rVert_{\cL^{p'} (S_\rho)}
 + C \bigl\lVert \mathrm{e}^{-\lambda x_d} \Delta u \bigr\rVert_{\cL^{p} (\mathbb{R}^d \setminus S_\rho)} .
 \end{align*}
 Since $C \lVert V \rVert_{\cL^{d/2} (S_\rho)} \leq 1/2$ we get
\[
 \bigl\lVert \mathrm{e}^{-\lambda x_d} u \bigr\rVert_{\cL^{p'} (S_\rho)}
 \leq 2 C \bigl\lVert \mathrm{e}^{-\lambda x_d} \Delta u \bigr\rVert_{\cL^{p} (\mathbb{R}^d \setminus S_\rho)} .
\]
We use $\mathrm{e}^{-\lambda x_d} \leq \mathrm{e}^{-\lambda \rho}$ for $x_d > \rho$ and obtain
\[
 \forall \lambda > 0 : \quad\bigl\lVert \mathrm{e}^{-\lambda (x_d-\rho)} u \bigr\rVert_{\cL^{p'} (S_\rho)}
 \leq 2 C \bigl\lVert  \Delta u \bigr\rVert_{\cL^{p} (\mathbb{R}^d \setminus S_\rho)} .
\]
Note that the right hand side of the last inequality is independent of $\lambda$ and that $x_d - \rho < 0$ on $S_\rho$.
Hence, if $u \not \equiv 0$ on $S_\rho$, we get a contradiction by choosing $\lambda$ large enough.
By our choice of $\rho$, we can iterate this procedure and find that $u \equiv 0$ on $\RR^d$.
\end{proof}
\end{example}
\begin{remark}[Weight functions]
The above example shows unique continuation on strips since the level sets of the weight function $\euler^{- \lambda x_d}$ are strips.
Hence, it might be tempting to search for radially symmetric weight functions in order to obtain unique continuation on annuli or balls.
Indeed, such weight functions have been used in many situations,  see the discussion in Remark~\ref{rem:BourgainK-05} below.
However, typically one wants the weight function to have nowhere vanishing gradient.
For radially symmetric functions, this poses a problem at the origin, which can be resolved in various ways.
One could exclude the origin from the domain and consider weight functions which are smooth except at the origin
cf.\ e.g.\ Remark~\ref{rem:BourgainK-05}, or use a two-weight Carleman inequality, cf.\ e.g.\ \cite{RodnianskiT-15}.
\end{remark}
\begin{example}[Elliptic differential operators]
Let $H$ be the elliptic partial differential operator
 \[
  H f(x) := \sum_{j,j=1}^d \frac{\partial}{\partial x_i} \left( a_{ij}(x) \frac{\partial}{\partial x_j} f(x) \right) + V(x) f(x),
 \]
acting on $C^2(\Omega)$, where $\Omega \subset \RR^d$ is open and connected, $V : \Omega \to \RR$ is bounded and measurable, the functions $a_{ij}: \Omega \to \RR$, $1 \leq i,j \leq d$, are Lipschitz continuous, $a_{ij} = a_{ji}$, and there is $\lambda > 0$ such that for all $x \in \Omega$ and all $\xi \in \RR^d$
  \[
   \frac{1}{\lambda} \lvert \xi \rvert^2 \leq \sum_{i,j = 1}^d a_{ij}(x) \xi_i \xi_j \leq \lambda \lvert \xi \rvert^2 .
  \]
 By means of Carleman estimates it has been shown that the class $\{f \in C^2(\Omega) \mid H f = 0 \}$ satisfies the strong unique continuation property, see for instance \cite{Wolff-93}, where more general results are discussed.
 One can generalize this result to Sobolev spaces $W^{2,2}(\Omega)$ or $W^{2,p}(\Omega)$, $p > 1$.
\end{example}
Next we supply an example which shows that the two properties from Definition~\ref{def:ucp} are actually distinct.
\begin{example}[Functions satisfying UCP, but not SUCP]
 Let $d \in \{ 3,4 \}$, $\Omega \subset \RR^d$ be open, $a_{i,j} : \Omega \to \RR$ Lipschitz for $i,j = 1, ..., d$, $A \in \cL^{d/2}_{\mathrm{loc}}(\Omega)$ and $B \in \cL^{d}_{\mathrm{loc}}(\Omega)$.
 Then solutions $u \in W^{2,2}_{\mathrm{loc}}(\Omega)$ of the differential inequality
 \begin{equation}
 \label{eq:differential_ineq}
  \Biggl\lvert \sum_{i,j = 1}^d a_{ij} \frac{\partial^2 u}{\partial x_i \partial x_j} \Biggr\rvert \leq A \lvert u \rvert + B \rvert \nabla u \rvert
 \end{equation}
satisfy the unique continuation property, but not necessarily the strong unique continuation property, see \cite{Wolff-93} and the references therein.
 Note that for $A,B \geq 0$, the set of solutions of the differential inequality \eqref{eq:differential_ineq} contains in particular solutions of the differential equation
 \[
  \sum_{i,j = 1}^d a_{ij} \frac{\partial^2 u}{\partial x_i \partial x_j} = A u + B \nabla u.
 \]
For other examples of this type see \cite{JerisonK-85,Kenig-86}.
\end{example}
\subsection{A result of Donnelly and Fefferman: Eigenfunctions of the Laplacian}
We now consider a $d$-dimensional, connected, compact manifold $M$ with a smooth (that is $C^\infty$) Riemannian metric.
The compactness will replace the condition of controlled growth at infinity of functions we have discussed in the context
of Hadamard's three circle theorem. We want to study differentiable functions on $M$.
\begin{example}
A prominent example of such a manifold $M$ is the $d$-dimensional torus $\TT^d := \RR^d / \ZZ^d$. Note that
 \[
  C^k(\TT^d) \cong C_{\mathrm{per}}^k(\RR^d) =
  \left\{
    u \in C^k(\RR^d)\mid u(x+k) = u(x)\ \text{for all}\ x \in \RR^d, k \in \ZZ^d
  \right\}.
 \]
 In particular, we can learn about periodic problems in Euclidean space by studying this example.
\end{example}
The following theorem quantifies the vanishing speed of solutions of the differential equation $-\Delta u = E u$, $E > 0$, where $\Delta$ denotes the Laplace-Beltrami operator on the manifold $M$. Here, the vanishing speed is quantified in $\cL^\infty$-norm. It can be found in Proposition~4.1 of \cite{DonnellyF-88}.
\begin{theorem}
\label{thm:DonnellyF1}
There are constants $C_1, C_2 \geq 0$, depending only on $d$, the diameter of $M$ and the maximum over all sectional curvatures on $M$ such that
for every $E > 0$, every $u : M \to \RR$, $0 \not\equiv u$ with $- \Delta u = E u$ on $M$,
and every $x_0 \in M$, $u$
can vanish at most of order $C_1 + C_2 \sqrt{E}$ with respect to the $\infty$-norm.

More precisely, for every $u \not\equiv 0$ with $- \Delta u = E u$ and every $x_0 \in M$ there is $\epsilon_0 > 0$ such that for every $\epsilon < \epsilon_0$, we have
\begin{equation}
\label{eq:DF}
 \epsilon^{C_1 + C_2 \sqrt{E}}
 \leq
\max_{x \in B(x_0, \epsilon)} \lvert u(x) \rvert
\end{equation}
and consequently
\[
 \lim_{\epsilon \to 0} \epsilon^{- \delta - (C_1 + C_2 \sqrt{E})} \max_{x \in B(x_0, \epsilon)} \lvert u(x) \rvert = \infty \text{  for all } \delta > 0.
\]
\end{theorem}
The balls are to be taken with respect to the geodesic distance on $M$.
\begin{remark}
 \begin{enumerate}[(i)]
  \item Even though we did not make any regularity assumption on $u$, by elliptic regularity theory, see for example
  \cite[Chapter 6.3]{Evans-98},  we know that any $u$ that solves the eigenvalue equation is in fact in $C^\infty(M)$.
  \item Vanishing with respect to the $\cL^\infty$-norm is a stronger statement than vanishing with respect to the $\cL^1$-norm as we have it in the definition of the strong unique continuation property.
  In fact, let $u$ vanish of order $C_1 + C_2 \sqrt{E}$ with respect to the $\cL^\infty$-norm. Then we have
  \[
   \int_{B(\epsilon)} \lvert u(x) \rvert \mathrm{d}x \leq \mathrm{Vol}(B(1)) \cdot \epsilon^d \max_{x \in B(\epsilon)} \lvert u(x) \rvert \leq \mathrm{Vol}(B(1)) \cdot \epsilon^{d + C_1 + C_2 \sqrt{E}}.
  \]
  However, for the property of {\em not vanishing} with respect to some order, the converse implication holds,
  so that Theorem \ref{thm:DonnellyF1} is a weaker statement than one about non-vanishing with respect to the $\cL^1$-norm.
\end{enumerate}
\end{remark}
Since $C_1$ and $C_2$ are not explicitly known, this theorem is most interesting for large $E$.
Inequality \eqref{eq:DF} controls some kind of local variation of $u$.
The higher $E$, the larger the variation of $u$ around a point $x_0$, cf. Example \ref{ex:trigonometric}.
It is natural to ask  whether one can complement inequality \eqref{eq:DF} by an upper bound.
This has been studied in \cite{DonnellyF-88} as well.
\begin{definition}
 Let $u \in C(M)$ be real-valued. The nodal set of $u$ is $N_u := \{ x \in M \mid u(x) = 0 \}$.
 We denote by $\mathcal{H}^{d-1}$ the $(d-1)$-dimensional Hausdorff measure on $M$.
\end{definition}
Recall that eigenfunctions of the Laplacian on an analytic manifold are analytic, see e.g. \cite[Theorem 7.5.1]{Hoermander-69}. By the theory of analytic sets, the nodal sets $N_u$ of such function have a well-defined Hausdorff measure $\mathcal{H}^{d-1}(N_u)$. The following theorem is due to \cite[Theorem 1.2]{DonnellyF-88}.
\begin{theorem}
\label{thm:DonnellyF2}
 Let $M$ be a compact, real-analytic, connected manifold (with real-analytic metric).
 Then, there exist $C_3$, $C_4$, depending on $M$, such that for every $u:M \to \RR$, $0 \not \equiv u$ and every $E \geq 0$ with $- \Delta u = E u$, we have
 \begin{equation}
 \label{eq:DF2}
  C_3 \sqrt{E} \leq \mathcal{H}^{d-1}(N_u) \leq C_4 \sqrt{E}.
 \end{equation}
\end{theorem}
\begin{example}[Vanishing speed and nodal sets of trigonometric functions]
\label{ex:trigonometric}
Let $M = \TT^1 = \mathbb{S}^1 \cong [0,1)$.
The eigenfunctions of the Laplace operator on $(0,1)$ with periodic boundary conditions are $\sin$ and $\cos$ waves. For simplicity we only study the eigenfunctions $u_n (x) = \sin (2 \pi n x)$ with corresponding eigenvalue $E_n = (2 \pi)^2 n^2$ and their vanishing speed at the point $x = 0$. For $\epsilon$ small, we have
\[
2 \pi n \epsilon \geq \sup_{x \in B(\epsilon)} \lvert u_n(x) \rvert \geq \frac{2 \pi n \epsilon}{2},
\]
thus in particular, since $\epsilon$ is small, $\sup_{x \in B(\epsilon)} \lvert u_n(x) \rvert \geq \epsilon^{2 \pi n} = \epsilon^{\sqrt{E_n}}$ whence inequality \eqref{eq:DF} holds with $C_1 = 0$ and $C_2 = 1$.
Furthermore, the $0$-dimensional Hausdorff measure of the zero set $N_{u_n}$ is the number of zeros of $u_n$ and we have $\mathcal{H}^0(N_{u_n}) = 2n$.
Thus, inequality \eqref{eq:DF2} holds with $C_3 = C_4 = 1/\pi$.
\end{example}
\begin{example}[Vanishing speed and nodal sets of spherical harmonics on $\mathbb{S}^2$] \label{ex:harmonics2}
We consider the real part of the spherical harmonics $Y_{l,l}$, $l \in \NN$, from Example~\ref{ex:spherical-harmonics}, i.e.
\[
 \Re Y_{l,l}  = \Re \bigl( c_l \cos (\theta)^l \exp (\mathrm{i} l \phi) \bigr)
          =  c_l \cos (\theta)^l  \cos (l\phi),
\]
where $\theta \in [-\pi / 2 , \pi / 2]$ and $\phi \in [0,2\pi)$. Recall that $-\Delta \Re Y_{l,l} = E_l \Re Y_{l,l}$ where $E_l = l (l+1)$. The function $\Re Y_{l,l}$ exhibits the highest order of vanishing at the poles $\theta =  \pm \pi / 2$ where its maximum behaves as $\lvert \pm \pi/2 - \theta \rvert^l$. Thus we have for all $x_0 \in \mathbb{S}^2$ and $\epsilon > 0$ sufficiently small
\[
 \max_{x \in B(x_0, \epsilon)} \Re Y_{l,l}(x) \geq \epsilon^{l} \geq \epsilon^{\sqrt{E_l}},
\]
where $B(x_0, \epsilon)$ denotes the ball with center $x_0$ and radius $\epsilon$ with respect to the geodesic distance. Thus Ineq.~\eqref{eq:DF} of Theorem~\ref{thm:DonnellyF1} holds with $C_1 = 0$ and $C_2 = 1$.
\par
Concerning Theorem~\ref{thm:DonnellyF2}, we note that the nodal set of $\Re Y_{l,l}$ consists of exactly $l$ meridians. It is of Hausdorff measure $\mathcal{H}^{d-1} (N_{\Re Y_{l,l}}) = 2\pi l$. Hence, Ineq.~\eqref{eq:DF2} of Theorem~\ref{thm:DonnellyF2} is satisfied with $C_3 = \pi$ and $C_4 = 2\pi$.
\end{example}
\subsection{A result of Kukavica: Eigenfunctions of Schr\"odinger operators}
Instead of eigenfunctions of the Laplacian, one can study solutions of the stationary Sch\"odinger equation $\Delta u = V u$ on a manifold $M$. In this setting the question arises how the vanishing order depends on properties of the potential $V$.
\par
Next we cite Theorem~5.2 of \cite{Kukavica-98}, which is a generalization of Theorem \ref{thm:DonnellyF1}.
\begin{theorem}
 \label{thm:Kukavica}
 Let $M$ be a compact, connected, smooth manifold of dimension $d$, let $V\in \cL^\infty(M)$ and $0 \not\equiv u \in W^{2,2}(M)$ with $\Delta u = V u$.
 Then there is a constant $C > 0$, depending only on $M$, such that $u$ can vanish at most of order
 \[
 C (1 + \lVert V \rVert_\infty^{1/2} + (\operatorname{osc} V)^2),
 \]
 where $\operatorname{osc}(V) := \sup V - \inf V$.
\par
 More precisely, for every $x_0 \in M$ and every $\epsilon > 0$ sufficiently small, we have
\begin{equation}
\label{eq:Kukavica}
 \epsilon^{C (1 + \lVert V \rVert_\infty^{1/2} + (\operatorname{osc} V)^2)}
 \leq
\max_{x \in B(x_0, \epsilon)} \lvert u(x) \rvert .
\end{equation}
\end{theorem}
In the case $V \equiv E$, this theorem reduces to Theorem \ref{thm:DonnellyF1}.
If we choose $V = E \cdot \chi_W$ where $\chi_W$ is the characteristic function of a open, non-empty, proper subset $W \subset M$
and $E$ a coupling constant, then the exponent in \eqref{eq:Kukavica} becomes $C( 1 + \sqrt{E} + E^2)$,
that is quadratic in $E$. Later, we will see similar statements with the better exponent $C( 1 + E^{2/3})$.
\par
Theorems \ref{thm:DonnellyF1}, \ref{thm:DonnellyF2}, and  \ref{thm:Kukavica}
apply to the  $d$-dimensional torus $\TT^d$.
This fits nicely to some partial differential equations in Euclidean space.
If one considers a cube $\Lambda \subset \RR^d$ as a domain and imposes periodic boundary conditions
on the solutions of the partial differential equation, then a problem on a torus results.
We will come back to discuss this situation in Sections
\ref{ss:eigenfunctions} and  \ref{ss:linear-combinations}.
\par
Theorem \ref{thm:Kukavica} can be reduced (by use of the exponential map) to a statement about elliptic operators on bounded domains $\Omega \subset \RR^d$, see \cite{Kukavica-98} for details.
\begin{theorem}
 Let $\Omega \subset \RR^d$ be open, connected and with $C^{1,1}$ boundary.
 Let $0 \not\equiv u \in W^{2,2}(\Omega)$ satisfy
 \[
  - \sum_{i,j = 1}^d \partial_i (a_{ij} \partial_j u) + V u = 0
 \]
where $V \in \cL^\infty(\Omega)$ and the $a_{ij} \in C^{0,1}(\overline \Omega)$, $i,j = 1, ..., d$
are uniformly Lipschitz continuous functions with $a_{ij} = a_{ji}$, satisfying the following ellipticity condition: there is $\lambda > 0$ such that for all $\xi \in \RR^d$ and all $x \in \Omega$ we have
\[
 \lvert \xi \rvert^2 / \lambda \leq \sum_{i,j = 1}^d a_{ij} (x) \xi_i \xi_j .
\]
Furthermore, we assume $\lvert a_{ij}(x) \rvert \leq \lambda$ and $\lvert \partial_k a_{ij}(x) \rvert \leq \lambda$ for almost all $x \in \overline \Omega$ and that $a_{ij}|_{\partial \Omega} \in C^{1,1}(\partial \Omega)$, $i,j,k = 1, ..., d$.
\par
Then there is a constant $C$, depending only on $d$, $\lambda$, the $C^{1,1}$-character of $\partial \Omega$
and the $C^{1,1}$-character of $a_{i,j}|_{\partial \Omega}$ such that at every $x_0 \in \Omega$, $u$ can vanish at most of order
 order $C (1 + \lVert V \rVert_\infty^{1/2} + (\operatorname{osc} V)^2)$.
\end{theorem}
%
%
%
%
%
%
%
%
\section{Retrieval of global features from local data}
\label{s:multiscale}
Let $\Lambda \subset \RR^d$ be open and connected and $W \subset \Lambda$.
One can ask the question whether it is possible to reconstruct certain properties of a function $f: \Lambda \to \CC$
only from data or certain features of $f$ on the set $W$. This might be possible,
if one has additional information on the regularity or rigidity of $f$.
\subsection{Rigidity of functions with concentrated Fourier transform}
A benchmark for reconstructing functions from partial data is the following theorem which is discussed in detail e.g.\ in \cite{ButzerSS-88}.
\begin{theorem}[Whittaker--Nyquist--Kotelnikov--Shannon sampling theorem]
 Let $f \in C(\RR) \cap \cL^2(\RR)$, such that the Fourier transform
 \[
  \hat{f} (p) = \frac{1}{\sqrt{2 \pi}} \int_{\RR} \euler^{-i x p} f(x) \mathrm{d}x
 \]
vanishes outside $[- \pi K, \pi K]$.
Then
 \[
 (S_K f) (x) = \sum_{j \in \ZZ} f(j/K) \frac{ \sin \pi (Kx - j)}{\pi (Kx - j)}
 \]
 converges absolutely and uniformly on $\RR$ and $S_K f = f$ on $\RR$.
\end{theorem}
One can relax the hypotheses and remove the compact support condition on $\hat{f}$. Then the aliasing error is estimated as
 \[
 \sup_{\RR} \lvert f - S_k f \rvert \leq \sqrt{\frac{2}{\pi}} \int_{\lvert p \rvert > \pi K} \lvert \hat{f}(p) \rvert \mathrm{d}p .
 \]
The Whittaker--Nyquist--Kotelnikov--Shannon sampling theorem allows one to reconstruct the complete function from
data on the discrete set $W = \{j/K \mid j \in \ZZ\} \subset \Lambda = \RR$.
This is due to the imposed rigidity requirement, which allows only for holomorphic functions. Next we formulate the Logvinenko-Sereda Theorem \cite{LogvinenkoS-74}, where an upper bound on the $\cL^p$-norm of a function is obtained from local data on an appropriately choosen subset $W \subset \RR$.
\begin{theorem}[Logvinenko-Sereda Theorem]
Let $\gamma, a >0$. Let $W \subset \RR$ be $(\gamma, a)$-thick, i.e.~$W$ is measurable and for all intervals $I\subset \RR$ of length $a$ we have
\[
 \lvert W \cap I \rvert \geq \gamma \cdot a .
\]
Let $p \in [1,\infty]$, $J\subset \RR$ be an interval of length $b>0$, and  $\psi \in \cL^p(\RR)$ with $\hat \psi $ supported in $J$.
Then there is a constant $C = C (ab , \gamma)$ such that
\[
 \lVert \psi \rVert_{\cL^p (W)} \geq C(ab,\gamma) \lVert \psi \rVert_{\cL^p (\RR)} .
\]
\end{theorem}
Note that the constant on the right hand side does not depend on the position of the interval $J$, nor on
detailed properties of the set $W$.
   Here $a$ plays the role of a scale and  $\gamma$ of a density.
Logvinenko and Sereda proved the statement with
$C(ab,\gamma)= \exp(-c (1+ab)/ \gamma )$,
while Kovrijkine showed  in \cite{Kovrijkine-01} that the constant  $C(ab,\gamma)$
can be chosen as a polynomial $( \gamma /c )^{c(1+ab)}$ of $\gamma$.
Here $c$  denotes a universal constant independent of the model parameters.
Furthermore, he showed the following refinement of Logvinenko-Sereda Theorem:
\begin{theorem}[Kovrijkine-Logvinenko-Sereda Theorem] \label{thm:Kovrijkine}
Let $\gamma, a >0$. Let $W\subset \RR$ be  $(\gamma, a)$-thick.
Let $p \in [1,\infty]$, $J_k\subset \RR$, $k=1, \ldots, s$ be intervals of length $b>0$, and  $\psi \in \cL^p(\RR)$ with $\hat \psi $ supported in $J= \cup_{k=1}^s J_k$.
Then
\[
 \lVert \psi \rVert_{\cL^p (W)} \geq C(ab,\gamma,s,p) \, \lVert \psi \rVert_{\cL^p (\RR)}
\]
with $C(ab,\gamma,s,p)= ( \gamma / c )^{ab(c/\gamma)^s+s-(p-1) / p}$
\end{theorem}
There exists a multidimensional analog,  for $\psi\in \cL^p(\RR^d)$, of the Logvinenko-Sereda Theorem as well, cf.~\cite{Kovrijkine-01,MuscaluS-13}.
The following consequence of Theorem~\ref{thm:Kovrijkine} is remarkable.
\begin{corollary} \label{obs:true}
Fix $\gamma,a,b>0, s \in\NN$.
Let $B\colon \cL^2(\RR)\to \cL^2(\RR)$ be the multiplication operator with the characteristic function of an $(\gamma,a)$-thick set.
For an interval $J$ of length $b$ set $\cF(J) = \{f \in \cL^2(\RR) \mid \supp \hat f \subset J\}$.
While $B$ is not injective, we have
\begin{equation} \label{eq:observation}
\lVert \psi \rVert_{\cL^2 (\RR)} \geq \lVert B \, \psi \rVert_{\cL^2 (\RR)} \geq \Big( \frac{\gamma}{c} \Big)^{ab(c/\gamma)^s+s-\frac{1}{2}} \, \lVert \psi \rVert_{\cL^2 (\RR)}
\quad
\text{for all}\
\psi \in \cup  +_{k=1}^s \cF(J_k)
\end{equation}
where $+_{k=1}^s \cF(J_k) = \operatorname{span}(\cF (J_1), \ldots \cF (J_s))$ and the union runs over all $s$-tuples $J_1, \ldots,J_s\subset \RR$ of intervals of length $b$ each.
\end{corollary}
None of the subspaces $\cF(J_k)$ has finite dimension, but they are all unitarily equivalent.
The constant $c$ in \eqref{eq:observation} in particular does not depend on the positions of the intervals $J_k$.
This resembles the definition of the uniform uncertainty principle  or restricted isometry property,
except for the fact that dimensions of all subspaces are infinite. Let us recall the  uniform uncertainty principle,
which plays a prominent role in compressed sensing and sparse recovery, cf.~for instance \cite{CandesRT-06,FoucartR-13}.
\begin{definition}\label{d:RUP}
Let $M, n,s\in \NN$, $B\colon \RR^M\to\RR^n$ be a linear map, and  $s \leq M$.
If
\begin{align*} 
(1-\delta_s) \lVert \psi \rVert^2 \leq \lVert B \psi \rVert^2 \leq  (1+\delta_s) \lVert \psi \rVert^2
\end{align*}
for all $\psi \in \RR^M$ with $\sharp\supp \psi \leq s$, then $\delta_s$ is called a \emph{restricted isometry constant}  (for $s$ and $B$),
and $B$ is said to satisfy an \emph{uniform uncertainty principle} or \emph{restricted isometry property}.
Here typically $M \gg n$.
\end{definition}
While this definition concerns finite matrices, the most interesting situation is when $M$ becomes very large,
and one wants an explicit control with respect to the dimension. This setting is then not too far form the infinite dimensional one.
The two-sided inequality \eqref{eq:observation} may be seen as an instance of the infinite dimensional analog to Definition
\ref{d:RUP}. This and multiscale versions of the Logvinenko-Sereda Theorem will be discussed in detail elsewhere.
\begin{example}[Spherical harmonics revisited]
Let us come back to Example \ref{ex:spherical-harmonics} of spherical harmonics discussed earlier.
In light of the Logvinenko-Sereda Theorem one can also ask the question how one has to choose observation sets $A_L \subset \mathbb{S}^2$, $L \in \NN$,
such that for all $L \in \NN$ one has an observability inequality which is uniform on
\[
 f \in \cF_L = \left\{ f \in \cL^2(\mathbb{S}^2) \mid f \in \mathrm{Span}\ \{ Y_{l,m} \mid l(l+1) < L, -l < m < l \} \right\},
\]
that is an inequality
\begin{equation}
 \label{eq:Ortega-CerdaP-13}
 \int_{\mathbb{S}^2} \lvert f \rvert^2 \leq C \int_{A_L} \lvert f \rvert^2 \quad \text{for all}\ f \in \cF_L
\end{equation}
with a constant $C > 1$ that does not depend on $L$.
\par
The answer is given by Theorem~1 of \cite{Ortega-CerdaP-13}. We formulate it reduced to the simpler $\mathbb{S}^2$ case.
\begin{theorem}[{Logvinenko-Sereda Theorem on the sphere}]
 A sequence of sets $A_L \subset \mathbb{S}^2$, $L \in \NN$, satisfies \eqref{eq:Ortega-CerdaP-13} if and only if there is $r > 0$ such that
 \begin{equation*} 
\Gamma = \Gamma_r :=
\inf_{L \in \NN} \inf_{z \in \mathbb{S}^2} \frac{ \mathrm{vol} \left( A_L \cap B(z, r/ \sqrt{L} \right)}{\mathrm{vol} \left( B(z, r/ \sqrt{L} \right)} > 0.
 \end{equation*}
The balls are to be taken with respect to the geodesic distance on $\mathbb{S}^2$.
\end{theorem}
Here $\Gamma$ plays the role of a density, while a space scale is provided by $ r/ \sqrt{L}$.
This implies in particular that for \eqref{eq:Ortega-CerdaP-13} to hold, we need that there is $r > 0$ such that for every $L \in \NN$ the complement
$A_L^c$ contains no $r / \sqrt{L}$-balls.
\end{example}
In the next section we will pursue the question which of the properties discussed so far survive
if the class of functions under consideration is not given by a Fourier condition, but by eigenfunctions of Schr\"odinger operators or linear combinations thereof.
This is a natural question, since the expansion in terms of eigenfunctions can be seen as an analogue or generalization of the Fourier transform.
\subsection{Eigenfunctions of Schr\"odinger operators}
\label{ss:eigenfunctions}
Recall that for $L > 0$, we write $\Lambda_L = (-L/2, L/2)^d$.
We assume that $\Lambda \in \{\RR^d , \Lambda_L\}$ and $W \subset \Lambda$ is an equidistributed subset of $\Lambda$.
To be more precise, given $G,\delta > 0$, we say that a sequence $z_j \in \RR^d$, $j \in (G \ZZ)^d$ is \emph{$(G,\delta)$-equidistributed}, if
 \[
  \forall j \in (G \ZZ)^d \colon \quad  B(\x_j , \delta) \subset \Lambda_G + j .
\]
Corresponding to a $(G,\delta)$-equidistributed sequence $z_j$ we define for $L \in G \NN$ the set
\[
W_\delta = \bigcup_{j \in (G \ZZ)^d } B(\x_j , \delta) \cap \Lambda ,
\]
see Fig.~\ref{fig:equidistributed} for an illustration. Note that the set $W_\delta$ depends on $G$ and the choice of the $(G,\delta)$-equidistributed sequence and, if $\Lambda = \Lambda_L$, also on the scale $L$.
\begin{figure}[ht]\centering
\begin{tikzpicture}
\pgfmathsetseed{{\number\pdfrandomseed}}
\foreach \x in {0.5,1.5,...,4.5}{
  \foreach \y in {0.5,1.5,...,4.5}{
    \filldraw[fill=gray!70] (\x+rand*0.35,\y+rand*0.35) circle (0.15cm);
  }
}
\foreach \y in {0,1,2,3,4,5}{
\draw (\y,0) --(\y,5);
\draw (0,\y) --(5,\y);
}

\begin{scope}[xshift=-6cm]
\foreach \x in {0.5,1.5,...,4.5}{
  \foreach \y in {0.5,1.5,...,4.5}{
    \filldraw[fill=gray!70] (\x,\y) circle (1.5mm);
  }
}
\foreach \y in {0,1,2,3,4,5}{
  \draw (\y,0) --(\y,5);
  \draw (0,\y) --(5,\y);
}
\end{scope}
\end{tikzpicture}
\caption{Illustration of $W_\delta$ within the region $\Lambda = \Lambda_5 \subset \mathbb{R}^2$ for periodically (left) and non-periodically (right) arranged balls.\label{fig:equidistributed}}
\end{figure}
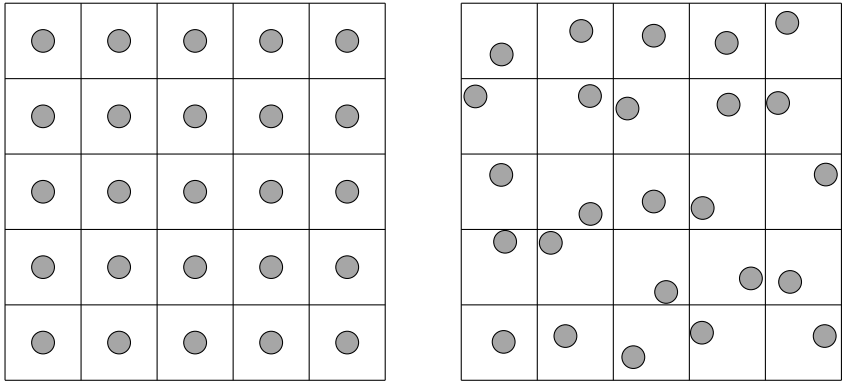
For a bounded and measurable potential $V : \RR^d \to \RR$ we introduce the self-adjoint Schr\"odinger operator $H :=-\Delta + V$ on $\cL^2 (\RR^d)$.
If $\Lambda=\RR^d$ then  $H_\Lambda$ coincides with $H$, if  $\Lambda = \Lambda_L$ for some finite $L$ then
 $H_\Lambda$ denotes the restriction of $-\Delta + V$ to $\cL^2 (\Lambda)$ with  Dirichlet, Neumann, or periodic boundary conditions.
Our aim is to prove $\lVert \psi \rVert_{\cL^2 (W_\delta)} \geq C_{} \lVert \psi \rVert_{\cL^2 (\RR^d)}$ for eigenfunctions $\psi$ of $H_\Lambda$,
with an  explicit and $L$-independent constant $C_{} > 0$. In the one-dimensional situation this problem reduces to an application
of Gronwall's inequality as carried out in \cite{Veselic-96,KirschV-02} for periodically arranged balls on the real line and in \cite{HelmV-07} for balls on metric graphs,
cf.\ Lemma~10 in the preprint \cite{HelmV-06-arxiv} for details.
We restate it here for $(1,\delta)$-equidistributed sequences.
\begin{lemma} \label{lemma:Gronwall}
Let $d = 1$.
For each $\delta \in (0,1/2)$ there is a constant $C_\delta > 0$,
such that for all $L \in 2\NN-1$ and $\Lambda \in \{\RR , \Lambda_L\}$, $V:\RR\to \RR$ measurable and bounded,
all $\psi \in W^{2,2} (\Lambda)$ satisfying $H_\Lambda \psi = E \psi$ for some $E \in \RR$,
all $(1,\delta)$-equidistributed sequences $z_j$, $j \in \ZZ$, and all $k \in \ZZ \cap \Lambda$ we have
\[
 \lVert \psi \rVert_{\cL^2 (B (\delta, z_k))} \geq C_{\mathrm{ucp}} \lVert \psi \rVert_{\cL^2 (\Lambda_1(k))}
 \quad \text{and} \quad
 \lVert \psi \rVert_{\cL^2 (W_\delta)} \geq C_{\mathrm{ucp}} \lVert \psi \rVert_{\cL^2 (\Lambda)} ,
\]
where
\[
 C_{\mathrm{ucp}} = \left (\left \lceil 1/\delta \right\rceil \euler^{ 2C_\delta + 2\lVert V - E \rVert_\infty } \right )^{-1}.
\]
\end{lemma}
Thus we are indeed considering inequalities of the type
\eqref{eq:uniform-very-strong-connection-ucp}
as discussed in Remark \ref{r:normalization}.
\begin{proof} For $k \in \Lambda \cap \ZZ$ set $f_k(x) = \lVert \psi \rVert_{\cL^2(B (x+z_k,\delta))}^2 > 0$ whenever $B (x+z_k , \delta) \subset \Lambda$.
By Sobolev norm estimates and the eigenvalue equation there is a $\delta$-dependent constant $C_\delta > 0$ such that
\begin{align*}
 \left\lvert \frac{\partial}{\partial x} f_k(x) \right\rvert
  &\leq 2 \lVert \psi \rVert_{\cL^2(B (x+z_k,\delta))} \lVert \psi' \rVert_{\cL^2(B (x+z_k,\delta))} \\
  &\leq 2 \left[ C_\delta + \lVert V - E \rVert_\infty \right] \lVert \psi \rVert_{\cL^2(B (x+z_k,\delta))}^2
     = 2\left[ C_\delta + \lVert V - E \rVert_\infty \right] f_k(x),
\end{align*}
see \cite{Veselic-96,KirschV-02} for details. Applying Gronwall's lemma, we obtain
\begin{align} \label{eq:gronwall}
 & f_k(x) \leq \euler^{2\left[ C_\delta + \lVert V - E \rVert_\infty \right] \lvert x \rvert} f_k(0)  \nonumber \\[1ex]
 \Leftrightarrow \quad &
 \lVert  \psi \rVert_{\cL^2(B (x+z_k,\delta))}^2 \leq
 \euler^{2\left[ C_\delta + \lVert V - E \rVert_\infty \right] \lvert x \rvert}\lVert \psi \rVert_{\cL^2(B (z_k,\delta))}^2 .
\end{align}
Positioning $x \in (-1,1)$ we cover $\Lambda_1(k)$ by $\left\lceil 1/\delta \right\rceil$ intervals of length $\delta$ and obtain
\[
 \lVert \psi \rVert_{\cL^2(\Lambda_1 (k))}^2
\leq \left\lceil 1/\delta \right\rceil \euler^{ 2C_\delta + 2\lVert V - E \rVert_\infty } \lVert \psi \rVert_{\cL^2(B (z_k , \delta))}^2 ,
\]
which proves the first inequality. The second inequality follows immediately by summing up the disjoint intervals $\Lambda_1(k)$, $k \in \ZZ \cap \Lambda$.
\end{proof}
Note that the constant $C_{\mathrm{ucp}}$ in Lemma~\ref{lemma:Gronwall} is independent of $L$.
For this reason we call an estimate of this type scale-free unique continuation principle.
The drawback of this result is that it is restricted to the one-dimensional situation.
Also, we did not track the explicit $\delta$-dependence.
\par
Now we turn to the multidimensional case. We start by recalling quantitative unique continuation estimates.
The following theorem from \cite{BourgainK-05} may be understood as an analogue of Theorems~\ref{thm:DonnellyF1}, \ref{thm:Kukavica}, and Ineq.~\eqref{eq:gronwall}
for Schr\"odinger operators on $\RR^d$.
\begin{theorem} \label{thm:BourgainK-05}
Let $\gamma, V\colon \RR \to \RR$ be bounded and measurable, and $u \colon \RR \to \CC$ a bounded solution of
 $\Delta u = V u + \gamma$ with $u (0) = 1$. Then there  are constants $c, c' \in (0,\infty)$,
such that  for all $x \in \RR^d$ we have
\begin{equation} \label{eq:BourgainK}
 \max_{\lvert y - x \rvert \leq 1} \lvert u (y) \rvert + \lVert \gamma \rVert_{\infty}
 >
 c \exp \left( -c' \lvert x \rvert^{4/3} \log \lvert x \rvert \right) .
\end{equation}
\end{theorem}
The proof is based on following Carleman estimate, see \cite{EscauriazaV-03,BourgainK-05}.
 \begin{theorem}
There are $\alpha_0,C > 1$ such that for all $\rho > 0$ there is $w_\rho: \mathbb{R}^d \to \mathbb{R}$
s.t. for all $\alpha \geq \alpha_0$ and $u \in W^{2,2} (\RR^d)$ with support in $B(\rho) \setminus \{0\}$ we have
   \begin{equation} \label{eq:CarlemanBK}
    \alpha^3 \int_{\mathbb{R}^d} w_\rho^{-1-2\alpha} u^2 \leq C_1 \rho^4 \int_{\mathbb{R}^d} w_\rho^{2-2\alpha} (\Delta u)^2
    \quad \text{and} \quad
    \frac{\lvert x \rvert}{\euler \rho} \leq w_\rho(x) \leq \frac{\lvert x \rvert}{\rho} .
  \end{equation}
  \end{theorem}
\begin{remark} \label{rem:BourgainK-05}
\begin{enumerate}[(i)]
 \item
 In fact, one can choose for $\rho >0$
 \begin{align*}
\varphi &\colon [0,\infty) \to [0,\infty),
&\varphi(s)&:=  s\cdot\exp \left(-\int_0^s \frac {1 - \rm e^{-t}} t \, d t\right)
\\
w_\rho&\colon \RR^d\to [0,\infty),
&w_\rho(x)&:= \varphi(\lvert x \rvert / \rho).
\end{align*}
Then $\varphi$ is a strictly increasing continuous function, on $(0, \infty)$ even smooth,
and
\begin{equation*}
\label{eq:carleman_condition}
\frac{\lvert x \rvert}{\rm e  \rho}  \leq w_\rho(x) \le \frac{\lvert x \rvert}{ \rho}
\quad\text{ for all } x \in B(0,\rho)
\end{equation*}
 \item
 The particular feature of this Carleman estimate is that the weight function is not exponential as, e.g., in Ineq.~\eqref{eq:CarlemanKRS}.
 Furthermore, the particular scaling with $\alpha$ is crucial to obtain the exponent $4/3$ in Ineq.~\eqref{eq:BourgainK}.
 \item Theorem~\ref{thm:BourgainK-05} was a crucial step for the answer on a long-standing problem in the theory of random Schr\"odinger operators, namely
 Anderson localization for the continuum Anderson model with Bernoulli-distributed coupling constants. Let us emphasize that the precise decay rate in Ineq.~\eqref{eq:BourgainK} was essential for this application.
 If, instead of Ineq.~\eqref{eq:BourgainK}, one would have at disposal only a slightly weaker version, where the exponent
 $4/3$ would be replaced by 1.35, one could not conclude localization for the continuum Anderson-Bernoulli model using the same techniques,
 cf.~\cite[p.~412]{BourgainK-05}.
\end{enumerate}
\end{remark}
There are local $\cL^2$-variants of Theorem~\ref{thm:BourgainK-05}, see \cite{GerminetK-13b, BourgainK-13, Rojas-MolinaV-13}.
As an exemple, we formulate Theorem~3.4 of \cite{BourgainK-13}.
\begin{theorem} \label{thm:BourgainK-13}
 Let $\Lambda \subset \RR^d$ be an open subset of $\RR^d$ and consider a real measurable and bounded function $V$ on $\Lambda$.
 Let $\psi \in W^{2,2}(\Lambda)$ be real-valued and $\zeta \in \cL^2(\Lambda)$ be defined by $-\Delta \psi + V \psi = \zeta$ almost everywhere on $\Lambda$.
 Let $\Theta \subset \Lambda$ be a bounded and measurable set where $\lVert \psi \rVert_{\cL^2(\Theta)} > 0$. Set
 \[
  {Q}(x, \Theta) := \sup_{y \in \Theta} \lvert y - x \rvert \quad \text{for}\ x \in \Lambda.
 \]
 Consider $x_0 \in \Lambda \setminus \overline\Theta$ such that ${Q} = {Q}(x_0, \Theta) \geq 1$,
$\operatorname{dist}(x_0, \Theta) >0$, and $B(x_0, 6 {Q} + 2 ) \subset \Lambda$. Then given $0 < \delta \leq \min \{ \operatorname{dist}(x_0, \Theta), 1/24 \}$, we have
 \begin{equation} \label{eq:BourgainKlein}
  \left( \frac{\delta}{{Q}} \right)^{K \bigl(1 + \lVert V \rVert_\infty^{2/3} \bigr) \left({Q}^{4/3} + \log \frac{ \lVert \psi \rVert_{\cL^2(\Lambda)}}{\lVert \psi \rVert_{\cL^2(\Theta)}} \right) }
  \lVert \psi \rVert_{\cL^2(\Theta)}^2
  \leq
  \lVert \psi \rVert_{\cL^2(B(x_0, \delta))}^2 + \delta^2 \lVert \zeta \rVert_{\cL^2(\Lambda)}^2 ,
 \end{equation}
where $K > 0$ is a constant depending only on $d$.
\end{theorem}
\begin{remark}
In the case $\zeta=0$ inequality \eqref{eq:BourgainKlein} estimates the quotient
\[
\frac{\lVert \psi \rVert_{\cL^2(\Theta)}}{\lVert \psi \rVert_{\cL^2(B(x_0, \delta))} }
\]
of two local $\cL^2$-norms in terms of another such quotient
\[
\frac{\lVert \psi \rVert_{\cL^2(\Lambda)}}{\lVert \psi \rVert_{\cL^2(\Theta)}}
\]
If an estimate on the latter is not provided a-priori, one might wonder, whether one is running in a vicious circle or
an induction without induction anchor. Indeed, for many applications the bound in Theorem \ref{thm:BourgainK-13},
and likewise the corresponding estimates in \cite{GerminetK-13b, Rojas-MolinaV-13}, need to be complemented by some other information.
This is quite analogous with the situation encountered in Example \ref{ex:powers} and Corollary \ref{cor:growth-vs-vanishing}.
Only when we are supplied with some estimate which controls the global growth of the function $f_k$,
we can say at what fastest rate it can vanish at the origin.
\end{remark}
\begin{remark}
Theorem~\ref{thm:BourgainK-13} is applied in \cite{BourgainK-13} to obtain bounds on the density of states outer measure
for Schr\"odinger operators in dimension $d \in \{1,2,3\}$.
The restriction on the dimension stems from the decay rate $4/3$ in Theorem~\ref{thm:BourgainK-13} and would be lifted if
the inequality \eqref{eq:BourgainKlein} would be at disposal with $Q^{4/3}$  replaced by $Q$.
However, in the case of complex-valued potentials Meshkov's example \cite{Meshkov-92} shows that it is not possible to improve the exponent $4/3$.
The example of Meshkov does not apply to real valued potentials.
However, at the moment it is not known how to exploit this additional property of the potential in order to obtain
improved quantitative unique countinuation estimates.
In particular,  an improvement of \eqref{eq:BourgainKlein}  must be based on some method different from Carleman estimates.
\end{remark}
Let us sketch the basic ideas of the proof of Theorem~\ref{thm:BourgainK-13} using the Carleman estimate~\eqref{eq:CarlemanBK}.
\begin{proof}[Sketch of proof of Theorem~\ref{thm:BourgainK-13}]
For simplicity, we restrict ourselves to the special case $\zeta \equiv 0$, $\Lambda = \RR^d$ and $x_0 = 0$. We cannot apply Ineq.~\eqref{eq:CarlemanBK} to $\psi$ directly, since $\psi$ is not supported in $B(r) \setminus \{0\}$ for some $r > 0$.
Therefore it is natural that a cut off function comes into play. We choose three annuli
\[
 A_1 = B (3\delta / 4) \setminus B (\delta / 4), \quad
 A_2 = B (2\euler Q) \setminus B (3\delta / 4) \quad\text{and}\quad
 A_3 = B (2\euler Q + 1) \setminus B (2 \euler Q) .
\]
and a cutoff function $\eta \in C_0^\infty (\RR^d;[0,1])$ as illustrated in Fig.~\ref{fig:cutoff},
with support in $B (2\euler Q + 1) \setminus B (\delta / 4)$ and the properties that
\begin{equation} \label{eq:ass2}
\begin{cases}
\max\{ \lvert \nabla \eta \rvert , \lvert \Delta \eta \rvert \} \leq \tilde\Theta_1 / \delta^2 =: \Theta_1 & \text{on}\ A_1, \\
\eta \equiv 1 & \text{on}\ A_2, \\
\max\{ \lvert \nabla \eta \rvert , \lvert \Delta \eta \rvert \} \leq \Theta_2 & \text{on}\ A_3 ,
\end{cases}
\end{equation}
for some constants $\tilde\Theta_1,\Theta_2 > 0$ which depend only on the dimension.
\begin{figure}[t]\centering
\begin{tikzpicture}[yscale=1.3,xscale=1.3]
 \draw[-latex] (-5,0)--(5,0);
 \draw[-latex] (0,-0.1)--(0,1.2);
 \draw[rounded corners=7pt]
 (-5,0)--(-4.3,0)--(-4.15,1)--(-0.63,1)--(-0.39,0)--(0.39,0)--(0.63,1)--(4.15,1)--(4.3,0)--(5,0);
 \draw[dotted, thick] (-0.25,-0)--(-0.25,1);
 \draw[dotted, thick] (0.25,-0)--(0.25,1);
 \draw (0.25,-0.4) node {$\frac{\delta}{4}$};
 \draw (-0.25,-0.4) node {$-\frac{\delta}{4}$};
 \draw[dotted, thick] (-0.75,-0)--(-0.75,1);
 \draw[dotted, thick] (0.75,-0)--(0.75,1);
 \draw (0.75,-0.4) node {$\frac{3\delta}{4}$};
 \draw (-0.75,-0.4) node {$-\frac{3\delta}{4}$};
 \draw[dotted, thick] (-2,-0)--(-2,1);
 \draw[dotted, thick] (2,-0)--(2,1);
 \draw (2,-0.4) node {$Q$};
 \draw (-2,-0.4) node {$-Q$};
 \draw[dotted, thick] (-4,-0)--(-4,1);
 \draw[dotted, thick] (4,-0)--(4,1);
 \draw (4,-0.2) node {$2eQ$};
 \draw (-4,-0.2) node {$-2eQ$};
 \draw[dotted, thick] (-4.5,-0)--(-4.5,1);
 \draw[dotted, thick] (4.5,-0)--(4.5,1);
 \draw (4.5,-0.5) node {$2eQ + 1$};
 \draw (-4.5,-0.5) node {$-2eQ + 1$};
\begin{scope}[yshift=-0.5cm]
 \draw (0.25,2) arc (0:180:0.25);
 \draw (0.75,2) arc (0:180:0.75);
 \draw[thin, dashed] (2,2) arc (0:180:2);
 \draw (4,2) arc (0:180:4);
 \draw (4.5,2) arc (0:180:4.5);
 \draw (0,2.45) node {$A_1$};
 \draw (0,4.75) node {$A_2$};
 \draw (0,6.25) node {$A_3$};
 \filldraw[fill = black!20!white] (1,2.3) rectangle (1.72,3.02);
 \draw (1.36, 2.66) node {$\Theta$};
\end{scope}
\end{tikzpicture}
\caption{Cutoff function $\eta$, annuli $A_1$, $A_2$, $A_3$ and the set $\Theta$\label{fig:cutoff}}
\end{figure}
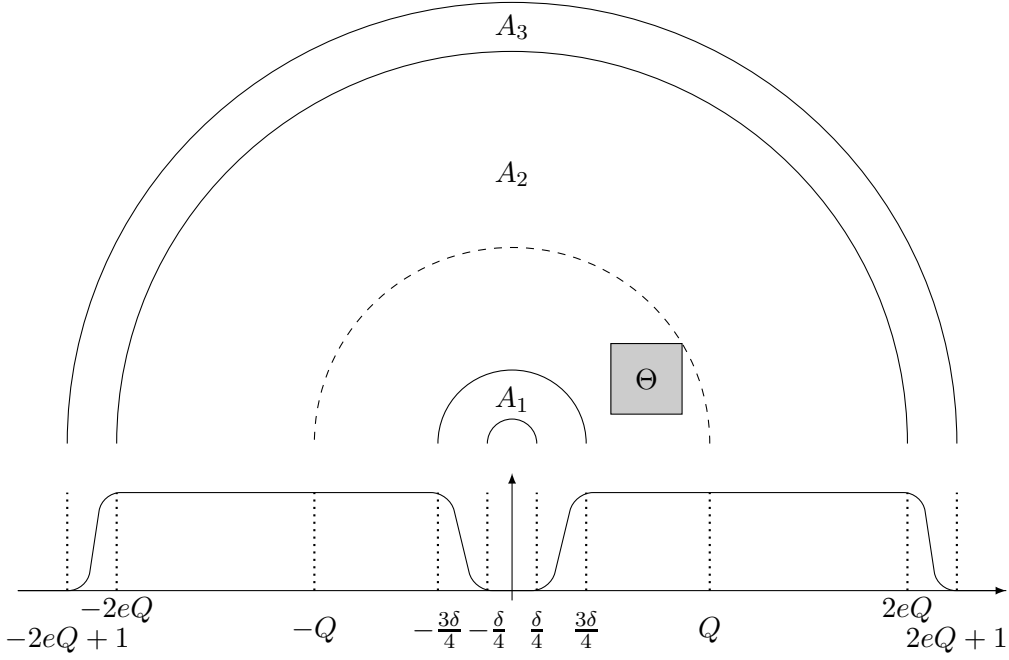
Note that by construction $\Theta \subset A_2 \cap B (Q)$.
Now we can apply Ineq.~\eqref{eq:CarlemanBK} with $\rho = 2\euler Q + 2$ to the function $u = \eta \psi$
and obtain using the product rule and $(a+b+c)^2 \leq 3 (a^2+b^2+c^2)$  and $\lvert \Delta \psi \rvert = \lvert V \psi \rvert$ that
\begin{multline*}
    \alpha^3 \int_{A_2} w_\rho^{-1-2\alpha} \psi^2
    \leq
    C_1 \rho^4 \int_{\mathbb{R}^d} w_\rho^{2-2\alpha} (\psi\Delta\eta + \eta\Delta \psi + 2 \left(\nabla \eta\right)^{\mathrm T} \nabla \psi)^2 \\
    \leq
    3 C_1 \rho^4 \left(\int_{A_1} + \int_{A_2} + \int_{A_3} \right) w_\rho^{2-2\alpha} (\psi^2 \lvert \Delta\eta \rvert^2 + \eta^2 \lVert V \rVert_\infty^2 \lvert  \psi \rvert^2 + 2 \lvert \nabla \eta \rvert^2 \lvert \nabla \psi \rvert^2) .
  \end{multline*}
  Since $w_\rho^{-1} \geq 1$ on $A_2$ we can replace the weight function on the left hand side by $w_\rho^{2-2\alpha}$.
  For the three integrals $\int_{A_i}$, $i \in \{1,2,3\}$, on the right hand side we proceed as follows.
  Since $\nabla \eta = \Delta \eta \equiv 0$ and $\eta \equiv 1$ on $A_2$ ,
  we can subsume the second integral on the right hand side into the left hand side by choosing $\alpha$ sufficiently large.
  For the first and the third integral we use our bound \eqref{eq:ass2} on the cutoff function and a Cacciopoli inequality
  to estimate $\int \lvert \nabla \psi \rvert^2$ by a constant (depending on $\delta$ and $\lVert V \rVert_\infty$) times $\int \lvert \psi \rvert^2$, see e.g.\ \cite{BourgainK-13} for details.
  Putting everything together we obtain
  \begin{equation} \label{eq:3annuli}
   \alpha^3 \int_{A_2}  w_\rho^{2-2\alpha} \psi^2 \lesssim \int_{A_1}  w_\rho^{2-2\alpha} \psi^2 + \int_{A_3} w_\rho^{2-2\alpha} \psi^2 ,
  \end{equation}
  up to a  multiplicative constant depending on $\delta$, $Q$, $\rho$, $\lVert V \rVert_\infty$, $\Theta_1$ and $\Theta_2$.
  Now we use that $\Theta \subset A_2 \cap B (Q)$, $A_1 \subset B(\delta)$ and our bounds on the weight function $(\rho / \lvert x \rvert)^{2\alpha - 2}  \leq w_\rho^{2-2\alpha}(x) \leq (\euler \rho / \lvert x \rvert)^{2\alpha - 2}$ on $B (\rho)$ to obtain
  \[
   \alpha^3 \left( \frac{\rho}{Q} \right)^{2\alpha-2} \int_{\Theta}   \psi^2
   \lesssim \left( \frac{4 \euler \rho}{\delta} \right)^{2\alpha-2} \int_{B (\delta)}  \psi^2
   + \left( \frac{\euler \rho}{2\euler Q} \right)^{2\alpha-2} \int_{\Lambda}  \psi^2  .
  \]
  If
  \[\alpha^3 2^{2 \alpha} \geq  2 \lVert \psi \rVert^2_{\cL^2 (\Lambda)} / \lVert \psi \rVert^2_{\cL^2 (\Theta)} ,
  \]
  we can subsume $\int_\Lambda \psi^2$ into the left hand side. The result follows by collecting all the constants.
\end{proof}
\begin{remark}
In Ineq.~\eqref{eq:3annuli} we estimate the values of the function $\psi$ on the middle annulus $A_2$
in terms of the values on the inner $A_1$ and outer $A_3$ annuli. Thus we have a similar geometric
situation as in Hadamard's three circle theorem \ref{t:hadamard3circle}.
\end{remark}
Quantitative unique continuation estimates as in Theorem~\ref{thm:BourgainK-13} are useful to obtain scale-free quantitative unique continuation estimates.
The following theorem was proven in \cite{Rojas-MolinaV-13} if $\Lambda = \Lambda_L$ and has been adapted to the case $\Lambda = \RR^d$ in \cite{TautenhahnV-15}.
It is a multidimensional analogue of Lemma~\ref{lemma:Gronwall} with an explicit dependence on $\delta$ and $\lVert V-E \rVert_\infty$.
\begin{theorem}
\label{t:scale_free-UCP}
Let $\Lambda \in \{\Lambda_L , \RR^d\}$. There exists a constant $K\in(0,\infty)$ depending merely on the dimension $d$,
such that for any $G >0$, $\delta \in (0,G/2]$, any $(G,\delta)$-equidistributed sequence $z_j$, $j \in (G \ZZ)^d$,
any measurable and bounded $V\colon {\RR^d}\to \RR$, any $L \in 2\NN - 1$ and any real-valued $\psi \in W^{2,2}(\Lambda)$ satisfying
$\lvert \Delta \psi \rvert \leq \lvert (V-E)\psi \rvert $ almost everywhere on $\Lambda$ we have
\begin{equation} \label{eq:Rojas-MolinaV-13}
\lVert \psi \rVert_{\cL^2(\Lambda_L)}
\geq
\lVert \psi \rVert_{\cL^2(W_\delta)}
\geq
\left( \frac{\delta}{G}\right)^{K(1+G^{4/3} \lVert V-E \rVert_\infty^{2/3})}
\lVert \psi \rVert_{\cL^2(\Lambda_L)} .
\end{equation}
\end{theorem}
Recall that $W_\delta$ denotes the union of $\delta$-balls around an equidistributed sequence.
In comparison to Theorem \ref{thm:DonnellyF1}
we have here no dependence on the diameter of the set $\Lambda_L$,
because we have not just one base point $x_0$, but an equidistributed sequence $z_j$, $j \in (G \ZZ)^d$.
\begin{remark}
Such estimates are called quantitative unique continuation estimates, or uncertainty principles, or observability estimates.
Since there is no dependence on $L\in  2\NN - 1$ the estimate is called scale-free
and the constant $C_{\sfuc} = ( \delta / G )^{K_0(1+G^{4/3} \lVert V-E \rVert_\infty^{2/3})}$
is called scale-free unique continuation constant.
\par
The dependence on the other paramters is also of interest.
Only the sup-norm $\lVert V \rVert_\infty$ of the potential enters, no knowledge of $V$ beyond this is used, in particular no regularity properties.
The constant $C_{\sfuc}$ is polynomial in $\delta$ and (almost) exponential in $\lVert V \rVert_\infty$.
\end{remark}
\begin{remark}
In order to prove Theorem \ref{t:scale_free-UCP} one uses Theorem 3.1 in \cite{Rojas-MolinaV-13}, which is very similar to Theorem
 \ref{thm:BourgainK-13} above. The roles played by the different sets are as follows: $\Lambda$ is the original finite or infinite
 cube on which the function $\psi$ is considered. $\Theta$ is a cube of side $62 \lceil \sqrt d\rceil$ centered at a lattice point
 $ k \in \Lambda \cap \ZZ^d$ inside the cube $\Lambda$. One should think of  $\Theta$ as a neighbourhood of a unit cube $\Lambda_1(k)$
 centered at the same $k$.
 The ball $B(x_0, \delta)$ is placed in (say the right) next-neighbour unit cube adjacent to $\Lambda_1(k)$.
 There is an issue with lattice sites $k$ near the boundary of $\Lambda$, but for the moment let uns consider the cae of periodic
 boundary conditions on the faces $\Lambda$. Then we can consider equivalently a partial differential equation on a torus (without boundary).
 Unfortunarely one does not have a priori information about the quotient $\lVert \psi \rVert_{\cL^2(\Lambda)} / \lVert \psi \rVert_{\cL^2(\Theta)}$.
 As discussed before, without this information the bound \eqref{eq:BourgainKlein} cannot be applied directly.
 \par
 It turns out that it is sufficient that the a priori bound holds in a certain averaged sense: not for all lattice points
 $ k \in \Lambda \cap \ZZ^d$ but just for those which 'carry most weight'. To make this precise the notion of
 \emph{dominating sites} is introduced in \cite{Rojas-MolinaV-13}. One uses the following obvious but useful observation:
\end{remark}
 \begin{lemma}[A reverse Markov inequality]
 Let $N,T\in \NN$ and $\mu$ be a probability measure on $\overline{N} := \{ 1, ..., N\}$.
 Set $\mathcal{A} := \{ n \in \overline{N} \mid \mu(n) \leq \frac{1}{T} \frac{1}{N} \}$. Then $\mu(A) \leq 1/T$.
 \end{lemma}
For details of the proof of theorem \ref{t:scale_free-UCP}
see  \cite{Rojas-MolinaV-13}.
\begin{remark}
If we are dealing with neither an eigenfunction $\psi$, nor a function which satisfies the inequality
$\lvert \Delta \psi \rvert \leq \lvert (V-E)\psi \rvert $, but with a linear combinations of eigenfunctions
there is no easy way to apply Theorem \ref{t:scale_free-UCP}.
As we will see there are (at least) two approaches how to deal with the problem:
 \begin{itemize}
 \item If the energy interval, which contains the relevant eigenvalues is small enough
 one can control the norm of $\zeta$ sufficiently well. The drawback is that only small energy intervals are allowed.
 \item Or one uses a more sophisticated argument to exploit the full power of Carleman estimates.
 This includes introducing an additional ghost dimension and using two different interpolation estimates based on Carleman estimates.
 \end{itemize}
All this will be discussed in the next section.
\end{remark}
\subsection{Spectral subspaces of Schr\"odinger operators}
\label{ss:linear-combinations}
In \cite{Rojas-MolinaV-13} the authors posed the open question whether Ineq.~\eqref{eq:Rojas-MolinaV-13} holds also for linear combinations of eigenfunctions, i.e.\ for $\phi \in \ran \chi_{(-\infty , E]} (H_\Lambda)$. This is equivalent to
\begin{equation*} \label{eq:uncertainty}
  \chi_{(-\infty,E]} (H_\Lambda) \, \chi_{W_\delta} \, \chi_{(-\infty,E]} (H_\Lambda)
  \geq
  C \chi_{(-\infty,E]} (H_L) ,
\end{equation*}
 with an explicit dependence of $C$ on the parameters $\delta$, $E$ and $\lVert V \rVert_\infty$.
 Here $\chi_{I} (H_\Lambda)$ denotes the spectral projector of $H_\Lambda$ onto the interval $I$.
 A partial answer, for short energy intervals, was given in \cite{Klein-13} in the finite volume case $\Lambda = \{\Lambda_L\}$
 and adapted to the case $\Lambda = \RR^d$ in \cite{TautenhahnV-15}.
\begin{theorem} \label{thm:Klein-13}
Let $\Lambda \in \{\RR^d , \Lambda_L\}$. There is $K = K (d)$ such that for all $E,G > 0$, $\delta \in (0,G/2)$, all $(G,\delta)$-equidistributed sequences $z_j$,
any measurable and bounded $V\colon {\RR^d}\to \RR$, any $L \in 2\NN - 1$ and all intervals $I \subset (-\infty , E]$ with
\[
 \vert I \rvert \leq 2 \gamma \quad \text{where} \quad
 \gamma^2 = \frac{1}{2G^4} \left(\frac{\delta}{G}\right)^{K \bigl(1+ G^{4/3}(2\lVert V\rVert_\infty + E)^{2/3} \bigr)} ,
\]
and all $\phi \in \ran \chi_{I} (H_\Lambda)$ we have
 \[
  \lVert \phi \rVert_{\cL^2 (W_\delta)}
  \geq G^4 \gamma^2 \lVert \phi \rVert_{\cL^2 (\Lambda)} .
 \]
\end{theorem}
A full answer to the above question, i.e.\ Theorem~\ref{thm:Klein-13} for arbitrary compact energy intervals $I\subset \RR$
has been given in \cite{NakicTTV-14},
while full proofs will be provided in \cite{NakicTTV-prep}.
\begin{theorem} \label{thm:NakicTTV}
Let $\Lambda = \Lambda_L$. There is $K = K(d)$ such that for all $G > 0$, all $\delta \in (0,G/2)$, all $(G,\delta)$-equidistributed sequences $z_j$, all measurable and bounded $V: \RR^d \to \RR$, all $L \in G\NN$, all $E \geq 0$ and all $\phi \in \mathrm{Ran} (\chi_{(-\infty,b]}(H_{\Lambda_L}))$ we have
 \begin{equation*}
\lVert \phi \rVert_{\cL^2 (W_\delta)}^2
\geq C_{\sfuc} \lVert \phi \rVert_{\cL^2 (\Lambda_L)}^2
\end{equation*}
where
\begin{equation*}
C_{\sfuc} = C_{\sfuc} (d, G, \delta ,  E  , \lVert V \rVert_\infty )
:=  \left(\frac{\delta}{G} \right)^{K \bigl(1 + G^{4/3} \lVert V \rVert_\infty^{2/3} + G \sqrt{E} \bigr)} .
\end{equation*}
\end{theorem}
Let us shortly discuss the ideas for the proof of Theorem~\ref{thm:NakicTTV}. By scaling it suffices to consider $G=1$ only.
Given $V:\RR^d \to \RR$ and $L \in \NN$ we denote by $\psi_k$, $k \in \NN$, the eigenfunctions of $H_{\Lambda_L}$ with corresponding eigenvalues $E_k$. Then given $E \geq 0$ each $\phi \in \mathrm{Ran} (\chi_{(-\infty,E]}(H_{\Lambda_L}))$
can be represented as
\begin{equation} \label{eq:phi}
 \phi = \sum_{\genfrac{}{}{0pt}{2}{k \in \NN}{E_k \leq E}} \alpha_k \psi_k  \quad \text{with} \quad
 \alpha_k = \langle \psi_k , \phi \rangle .
\end{equation}
Let $R = \lceil 18\euler\sqrt{d} \rceil$. Using reflections and translations, we extend the eigenfunctions and the potential $V_L = V|_{\Lambda_L}$
in such a way to $\Lambda_{RL}$ that the extensions still solve the eigenvalue equation. We use the same symbols $V_L$ and $\psi_k$ for the extended versions.
This is possible for periodic, Dirichlet, and Neumann boundary conditions.
Let further $F : X = \Lambda_{RL} \times \RR \to \CC$ be defined by
\begin{equation} \label{eq:F}
 F (x , x_{d+1}) = \sum_{\genfrac{}{}{0pt}{2}{k \in \NN}{E_k \leq b}} \alpha_k \phi_k (x) \funs_k( x_{d+1}) ,
\end{equation}
where $s_k : \RR \to \RR$ is given by
\[
\funs_k(t)=\begin{cases}
	\sinh(\lambda_k t)/\lambda_k, & E_k>0,\\
	t, & E_k=0,\\
	\sin(\lambda_k t)/\lambda_k, & E_k<0,
\end{cases}
\]
with $\lambda_k = \sqrt{\lvert E_k \rvert}$. The function $F$ fulfills
\begin{equation*} \label{eq:DeltaF}
\Delta F = \sum_{i=1}^{d+1} \partial^2_{i} F  =  V_L F \quad \text{on} \quad  \Lambda_{R L} \times \RR
\end{equation*}
and
\begin{equation*} \label{eq:F-phi}
\partial_{d+1} F (\cdot , 0) = \sum_{\genfrac{}{}{0pt}{2}{k \in \NN}{E_k \leq b}} \alpha_k \psi_k (\cdot) \quad \text{on} \quad \Lambda_{R L}  .
\end{equation*}
In particular, for all $x \in \Lambda_L$ we have $\partial_{d+1} F (\cdot , x) = \phi(x)$.
This way we recover the original function we are interested in. Let $X_1 = \Lambda_L \times [-1,1]$ and $X_3 = \Lambda_{L + 18\euler\sqrt{d}} \times [- 9\euler\sqrt{d} , 9\euler\sqrt{d}]$.
The goal is to obtain lower and upper bounds on the $H^1$-norm of $F$, more precisely
\begin{equation} \label{eq:sandwich}
 D_1 \lVert \phi \rVert_{\cL^2 (\Lambda_L)} \leq \lVert F \rVert_{H^1 (X_3)}
 \leq
 D_2 \lVert \phi \rVert_{\cL^2 (W_\delta)}
\end{equation}
with explicit constants $D_1$ and $D_2$ independent on the scale $L$ and explicit in all the other parameters. The lower bound is a calculation
using the way how the sets $\Lambda_L$ and $X_3$ are chosen.
For the upper bound we use two different Carleman estimate, namely  Ineq.~\eqref{eq:CarlemanBK}
and Proposition~1 in the appendix of \cite{LebeauR-95}, and conclude two interpolation inequalities for the function $F$.
The two interpolation inequalities read as follows with explicitly controlable constants $D_3$, $D_4$ and suitable sets $U_1 \subset U_3 \subset X_3$, see \cite{NakicTTV-prep} for details on how $U_1$, $U_3$ and $X_3$ are chosen.
\begin{proposition} \label{prop:interpolation1}
For all $\delta \in (0,1/2)$, all $(1,\delta)$-equidistributed sequences $z_j$, all measurable
and bounded $V: \RR^d \to \RR$, all $L \in 2 \NN -
1$, all $E \geq 0$ and all
$\phi, F$ as in \eqref{eq:phi} and \eqref{eq:F} we have
\[
  \lVert F \rVert_{H^1 (U_1)} \leq D_3 \lVert (\partial_{d+1} F)(\cdot , 0) \rVert_{\cL^2 (W_\delta)}^{1/2} \lVert F \rVert_{H^1 (U_3)}^{1/2} .
\]
\end{proposition}
\begin{proposition} \label{prop:interpolation2}
For all $\delta \in (0,1/2)$, all $(1,\delta)$-equidistributed sequences $z_j$, all measurable and bounded $V: \RR^d \to \RR$,
all $L \in 2 \NN - 1$, all $E \geq 0$ and all
$\phi, F$ as in \eqref{eq:phi} and \eqref{eq:F} we have
 \[
\lVert  F \rVert_{H^1 (X_1)}
\leq D_4 \lVert  F \rVert_{H^1 (U_1)}^{\gamma} \lVert  F \rVert_{H^1 (X_3)}^{1- \gamma} .
\]
\end{proposition}
Let us now show how these two interpolation inequalities are applied to obtain the announced upper bound \eqref{eq:sandwich}.
Again, a calculation shows $\lVert F \rVert_{H^1 (X_3)}  \leq D_5 \lVert F \rVert_{H^1 (X_1)}$. Applying both interpolation inequalities we conclude
\[
 \lVert F \rVert_{H^1 (X_3)} \leq D_5 D_4 D_3 \lVert F \rVert_{H^1 (X_3)}^{1- \gamma}  \lVert  (\partial_{d+1} F)(\cdot , 0) \rVert_{\cL^2 (W_\delta)}^{\gamma / 2} \lVert F \rVert_{H^1 (U_3)}^{\gamma /2} .
\]
Since $U_3 \subset X_3$ we find
\[
 \lVert F \rVert_{H^1 (X_3)} \leq (D_5 D_4 D_3)^{2/\gamma}  \lVert (\partial_{d+1} F)(\cdot , 0) \rVert_{\cL^2 (W_\delta)} .
\]
Since $\partial_{d+1} F (\cdot , 0) = \phi$ this provides the upper bound
and the result follows by estimating carefully all the constants $D_i$, $i \in \{1,\ldots , 5\}$, and $\gamma$.

In a more elementray setting this startegy of proof has been developed already in \cite{JerisonL-99}.
Additionally, \cite{NakicTTV-prep} uses ideas from \cite{GerminetK-13b,Rojas-MolinaV-13}.
%
%
%
%
%
%
\section{Applications}
\subsection{Random Schr\"odinger operators}
This section is concerned with Schr\"odinger operators with random potential.
Such operators serve as quantum mechanical models of disordered condensed matter.
Spectral and analytical properties of solutions of corresponding elliptic partial differential equation
are studied in order to gain insight in the evolution behaviour of solutions of the corresponding
time dependent Schr\"odinger equation. This in turn allows for conclusions concerning the transport properties of the modelled material.
The most studied type of random Schr\"odinger operator is the alloy model, also called continuum Anderson model.
We will be concerned with a different type of random operator, namely the random breather model.
It is analytically more challenging, due to the non-linear influence of the random variables.
In the mathematical literature random breather potentials have been has been first considered in \cite{CombesHM-96}, and studied in \cite{CombesHN-01} and
 \cite{KirschV-10}. However, all these papers assumed unnatural regularity conditions, excluding the most basic and standard
 type of single site potential, where $u$ equals the characteristic function of a ball or a cube.
For more details see \cite{NakicTTV-14,NakicTTV-prep}.
\par
Consider a sequence $\omega = (\omega_j)_{j \in \ZZ^d}$ of positive, independent and identically distributed random variables.
We assume that the distribution measure $\mu$ of $\omega_j$ is supported in an interval $[\omega_-, \omega_+]$ satisfying
$0 \leq \omega_{-} < \omega_{+} < 1/2$.
The \emph{standard random breather potential} is the function
\begin{equation*}
\label{eq:ballRBP}
 V_\omega(x) = \sum_{j \in \ZZ^d} \chi_{B(j,\omega_j)} (x).
\end{equation*}
while the family $(H_\omega)_\omega$ with $H_{\omega} := -\Delta +V_\omega$ on $\RR^d$ is called \emph{standard random breather model}.
Note that the random potential is non-negative and uniformly bounded, and thus the operator $H_\omega$ is self-adjoint for almost every $\omega \in \Omega$.
We also define for $L \in \NN$ the operator $H_{\omega, L}$ as the restriction of $H_\omega$ onto $\Lambda_L$ with Dirichlet boundary conditions.
$H_{\omega,L}$ is a lower semi-bounded operator with compact resolvent.
Hence its spectrum consists of an infinite sequence of (random) isolated eigenvalues of finite multiplicity $E_1^L \leq E_2^L \leq E_3^L \leq \ldots$.
\par
Due to ergodicity the spectrum of the random operator $H_\omega$ on the full space is deterministic.
This means that there is $\Sigma \subset \RR$ such that $\sigma(H_\omega) = \Sigma$, almost surely.
Analogous statements hold for the absolutely continuous, the singular continuous, and the pure point part of the spectrum.
For most truly random models the singular continuous component of the spectrum is empty,
so the prominent question is to determine whether in a certain energy region the Schr\"odinger operator exhibits
pure point or absolutely continuous spectrum, corresponding to localized or delocalized states.
A mixture of both types of spectrum in the same energy region would be considered as a physical anomaly.
In what we want to discuss, a central quantity is the integrated density of states (IDS)
or spectral distribution function $N(E)$.
It is a function of the energy and measures the number of energy states per unit volume up to that energy.
The definition is as follows
\[
 N(E) := \lim_{L \to \infty} \frac{ \EE \left[ \mathrm{Tr} \left[ \chi_{(- \infty, E]}(H_{\omega,L}) \right] \right]}{L^d}, \quad E \in \RR .
\]
A priori it is not clear whether the limit exists but in many situations,
namely when the family of random operators is ergodic, as is the case here, this is a consequence of ergodic theorems.
See the monographs \cite{Stollmann-01,Veselic-08} for more details and further references.
\par
We are interested in Wegner estimates, that are estimates on the expected number
of eigenvalues within an interval $[E- \epsilon, E + \epsilon]$ in terms of $\epsilon$ and $L^d$, the volume of $\Lambda$.
Such estimates play an important part role in proving localization,
that is the almost sure existence of pure point spectrum of $H_\omega$ near the bottom of $\Sigma$.
Moreover, our Wegner estimate implies that the integrated density of states is H\"older continuous.
\par
In order to prove a Wegner estimate, we need to understand how the eigenvalues $E_n^L$, $n \in \NN$ of $H_{\omega, L}$ behave if we increase all $\omega_j$ by a small amount $\delta > 0$.
We use the notation $H_{\omega + \delta, L}$ for the operator $H_{\omega, \delta}$ where all $\omega_j$ have been replaced by $\omega_j + \delta$.
\begin{lemma}[Eigenvalue lifting for the standard random breather model]
	Let $H_{\omega,L}$ be as above and assume that $\omega \in [\omega_{-}, \omega_{+}]^{\ZZ^d}$, $\delta \leq 1/2-\omega_+$.
	Then, for all $L \in \NN$ and all $n \in \NN$ with $E_n^L(\omega) \in (- \infty , E_0]$ we have
	\begin{equation*}
	E_n^L(\omega + \delta) \geq E_n^L(\omega) + \left(\frac{\delta}{2} \right)^{\bigl[K \bigl(2+ \lvert E_0+1
	\rvert^{1/2} \bigr)\bigr]},
	\end{equation*}
	where $K$ is the constant from Theorem~\ref{thm:NakicTTV}. In particular, $K$ does not depend on $L$.
\end{lemma}
\begin{proof}
The function $V_{\omega + \delta} - V_\omega$ is the characteristic function of a disjoint union of annuli each of which has width $\delta$, see Figure~\ref{fig:annuli}. Every such annulus contains a ball of radius $\delta/2$, see Figure~\ref{fig:annuli} whence we have $V_{\omega + \delta} - V_\omega \geq \chi_{W_{\delta/2}}$ where $\chi_{W_{\delta/2}}$ is the characteristic function of $W_{\delta/2}$, a union of $\delta$-balls, centered at a $(1,\delta)$-equidistributed sequence.
\begin{figure}\centering
\begin{tikzpicture}
\draw (-0.1, -0.1) grid (5 + 0.1, 5 + 0.1);
\foreach \x in {0,...,4}{
	\foreach \y in {0, ..., 4}{
		\pgfmathsetmacro{\w}{0.15*(rand+1)};
			\filldraw[thick, pattern=north east lines,  even odd rule] (\x + 0.5, \y + 0.5 ) circle (\w + 0.2) -- (\x + 0.5, \y + 0.5) circle (\w);	
}}		
\begin{scope}[xshift=7cm,scale = 2]
\filldraw[thick, fill = black!70] (0.1, 2.25) circle (0.1);
\draw (0.5, 2.25) node[right] {$\chi_{W_{\delta/2}}$};
\filldraw[thick, pattern=north east lines] (0,1.65) rectangle (0.4,1.85);
\draw (0.5, 1.75) node[right] {$V_{\omega + \delta} - V_\omega$};
\draw (-0.05, -0.05) grid (2.05, 1.05);
\filldraw[thick, pattern=north east lines,  even odd rule] (0.5, 0.5) circle (0.45) -- (0.5, 0.5) circle (0.25);
\filldraw[thick, pattern=north east lines,  even odd rule] (1.5, 0.5) circle (0.3) -- (1.5, 0.5) circle (0.1);
\filldraw[thick, fill = black!70] (0.5, 0.15) circle (0.1);
\filldraw[thick, fill = black!70] (1.60, 0.67320508) circle (0.1);
\end{scope}
\end{tikzpicture}
%
\caption{Illustration of the increments $V_{\omega + \delta} - V_\omega$ and the choice of $W_{\delta/2}$\label{fig:annuli}}
\end{figure}
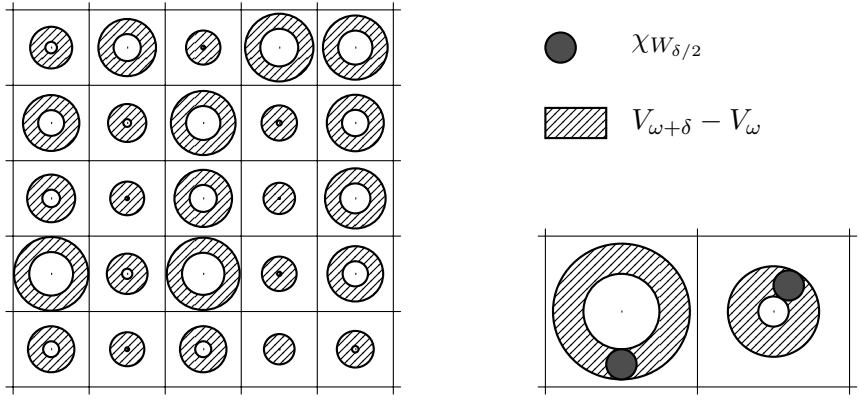
We denote the eigenfunctions, corresponding to $E_i^L(\omega + \delta)$ by $\phi_i^L$, $i \in \NN$.
Since $E_n^L(\omega + \delta) \leq E_n^L(\omega) + 1 \leq E_0 + 1$, we have by Theorem \ref{thm:NakicTTV} for all $\phi \in \mathrm{Span}\{\phi_1, ..., \phi_n\}$ with $\lVert \phi \rVert = 1$
\[
\left\langle \phi , \chi_{W_{\delta/2,L}} \phi \right\rangle \geq \left(\frac{\delta}{2} \right)^{\bigl[K \bigl(2+ \lvert E_0+1 \rvert^{1/2} \bigr)\bigr]}.
\]
Using this and the variational characterization of eigenvalues we estimate
		\begin{align*}
		E_n^L(\omega + \delta)
		& = \left\langle \phi_n, H_{\omega+\delta,L}\phi_n \right\rangle\\
		& = \max_{\phi \in \mathrm{Span}\{\phi_1, ..., \phi_n\}, \lVert \phi \rVert = 1} \left[ \left\langle \phi, H_{\omega,L}\phi \right\rangle + \left\langle \phi , \left( V_ {\omega + \delta,L} - V_{\omega,L} \right) \phi \right\rangle \right]\\
		& \geq \max_{\phi \in \mathrm{Span}\{\phi_1, ..., \phi_n\}, \lVert \phi \rVert = 1} \left[ \left\langle \phi, H_{\omega,L}\phi \right\rangle + \left\langle \phi , \chi_{W_{\delta/2,L}} \phi \right\rangle \right]\\
		& \geq \inf_{\mathrm{dim} \mathcal{D} = n} \max_{\phi \in \mathcal{D}, \lVert \phi \rVert = 1} \left[ \left\langle \phi, H_{\omega,L}\phi \right\rangle + \left(\frac{\delta}{2} \right)^{\bigl[K \bigl(2+ \lvert E_0+1 \rvert^{1/2} \bigr)\bigr]} \right]\\
		& = E_n^L(\omega) + \left(\frac{\delta}{2} \right)^{\bigl[K \bigl(2+ \lvert E_0+1 \rvert^{1/2} \bigr)\bigr]}.
		\qedhere
		\end{align*}
\end{proof}
Combining this Lemma with the method from \cite{HundertmarkKNSV-06} that was developed for
random Sch\"odinger operators with alloy type potential we obtain in \cite{NakicTTV-14, NakicTTV-prep} a Wegner estimate for the standard random breather model.
\begin{theorem}[Wegner estimate for the standard random breather model]
\label{thm:Wegner}
Assume that $\mu$ has a
bounded density $\nu$ supported in $[\omega_{-}, \omega_{+}]$ with $0 \leq \omega_{-} < \omega_{+} < 1/2$.
Fix $E_0 \in \RR$. Then there are
$C=C(d,E_0)$ and $\epsilon_{\max}=\epsilon_{\max}(d,E_0  ,\omega_{+}) \in (0,\infty)$
such that for all $\epsilon \in (0,\epsilon_{\max}]$ and $E \geq 0$ with
$[E-\epsilon, E+\epsilon] \subset (- \infty , E_0]$,
we have
\begin{equation*}
\EE \left[ \mathrm{Tr} \left[ \chi_{[E- \epsilon, E + \epsilon]}(H_{\omega,L}) \right] \right]
\leq
C
\lVert \nu \rVert_\infty
\epsilon^{[K(2+{\lvert E_0 + 1 \rvert}^{1/2})]^{-1}}
\left\lvert\ln \epsilon \right\rvert^d L^d
\end{equation*}
where $K$ is the constant from Theorem \ref{thm:NakicTTV}.
The constant $\epsilon_{\max}$ can be chosen as
\begin{equation*}
\epsilon_{\max} =
\frac{1}{4}
\left( \frac{1/2 - \omega_{+}}{2} \right)^{K(2+{\lvert E_0+1 \rvert}^{1/2})}.
\end{equation*}
\end{theorem}
Here $\EE$ denotes the expectation w.r.t.~the random variables $\omega_j, j\in \ZZ^d$.
From our Wegner estimate, we can deduce that the IDS is locally H\"older continuous.
\begin{corollary}[H\"older continuity of the IDS]
For every $E_0 \in \RR$ there are a constants $\tilde C,c > 0$ such that for all $E_1 < E_2 \leq E_0$ we have
\[
 \lvert N(E_2) - N(E_1) \rvert \leq C \cdot \lvert E_2 - E_1 \rvert^c .
\]
\end{corollary}
\begin{proof}
 For every $L \in 2 \NN -1$ we have
 \begin{align*}
  &\quad \frac{ \lvert \EE \left[ \mathrm{Tr} \left[ \chi_{(- \infty, E_2 ]}(H_{\omega,L}) \right] \right] - \EE \left[ \mathrm{Tr} \left[ \chi_{(- \infty, E_1 ]}(H_{\omega,L}) \right] \right] \rvert}{L^d}
   \leq
  \frac{ \EE \left[ \mathrm{Tr} \left[ \chi_{[E_1, E_2 ]}(H_{\omega,L}) \right] \right]}{L^d} \\
  & \leq C
\lVert \nu\rVert_\infty
\left\lvert \frac{E_2 - E_1}{2}\right\rvert^{[K(2+{\lvert E_0 + 1 \rvert}^{1/2})]^{-1}}
\cdot
\left\lvert\ln \frac{E_2 - E_1}{2} \right\rvert^d \\
& \leq \tilde C \lvert E_2 - E_1 \rvert^c.
\qedhere
 \end{align*}
\end{proof}
 \begin{remark}
 In \cite{NakicTTV-prep,TaeuferV} we establish the Wegner bound for a much more general class of random potentials.
 Here, for the sake of simplicity, we have restricted ourselves to the case of the standard random breather model.
 \end{remark}
In what we presented so far, the scale-free unique continuation principle was used to remove the so called \emph{covering condition}. In fact,  this condition featured in many older results on Wegner estimates, see for instance the original papers \cite{Kirsch-96,CombesH-94} or the detailed discussion in the monograph \cite{Veselic-08}. Since the covering conditions plays a role in other types of results on spectral properties of random Schr\"odinger operators, the scale-free unique continuation principle is a promising tool beyond just proofs of Wegner estimates. For instance, results of Shirley \cite{Shirley-14-dissertation} on Minami estimates and spectral statistics of one-dimensional models use the covering condition  as well.
It is natural to conjecture that the scale-free unique continuation principle can be used to remove this assumption.  Indeed, this has been carried out in the recent paper
\cite{Shirley-15}, see Theorem 1.1 there. It uses the  scale-free unique continuation principle  of \cite{NakicTTV-14} for one-dimensional configuration space, see
\cite[Theorem 4.1]{Shirley-15}.
\subsection{Control of the heat equation}
The aim here is to study in a multiscale geometry the control cost for the heat equation,
i.e. the infimum over $\cL^2$-norms of control functions which drive a system to zero at a prescribed time $T > 0$.
\par
We consider the controlled heat equation
\begin{equation} \label{eq:control_problem_concrete}
\begin{cases}
\partial_t u - \Delta u + Vu = f\chi_{W}, & u \in \cL^2([0,T] \times \Lambda),  \\
u = 0, & \text{on}\ (0,T) \times \partial \Lambda ,  \\
u(0,\cdot) =u_0, & u_0\in \cL^2(\Lambda),
 \end{cases}
 \end{equation}
where $\Lambda = \Lambda_L$ is a $d$-dimensional cube of side length $L \in \NN$ and $W$ is a union of $\delta$-balls within $\Lambda$, arizing from a $(1,\delta)$-equidistributed sequence.
In \eqref{eq:control_problem_concrete} $u$ is the state and $f$ is the control function which acts on the system through the control set $W \subset \Lambda$.
\par
We say that the system \eqref{eq:control_problem_concrete} is null controllable at time $T > 0$,
if there is for each initial state $u_0\in \cL^2(\Lambda)$ a control function $f \in \cL^2([0,T] \times W)$ such that the corresponding solution of \eqref{eq:control_problem_concrete} is zero at time $T$.
It is known, see for instance \cite{Fursikov-96} that the system \eqref{eq:control_problem_concrete} is null controllable at any time $T > 0$.
However, we want to estimate the cost, that is the $\cL^2$-norm of the control function $f \in \cL^2([0,T] \times W)$ in relation to the norm of the initial state $u_0$.
\par
The controllability cost $\mathcal{C}(T,u_0)$ at time $T$ for the initial state $u_0$ is given by
\[
\mathcal{C}(T,u_0) = \inf \left\{ \lVert f \rVert_{\cL^2([0,T]\times \omega )}\mid
u \
\text{is solution of \eqref{eq:control_problem_concrete} and }
u(T,\cdot)=0  \right\}.
\]
Combining Theorem~\ref{thm:NakicTTV} with results from \cite{Miller-10} one finds the following result, see \cite{NakicTTV-prep} for details.
\begin{theorem}
\label{thm:contcost}
For every $G > 0$, $\delta \in (0, G/2)$ and $K_V \geq 0$ there is $T' = T'(G, \delta, K_V) > 0$ such that
for all $T \in (0,T']$, all $(G,\delta)$-equidistributed sequences, all measurable and bounded $V: \RR^d \to \RR^d$ with $\lVert V \rVert_\infty \leq K_V$ and all $L \in G \NN$, the system \eqref{eq:control_problem_concrete} is null controllable on the set $W$ with cost $\mathcal{C}(T,u_0)$ satisfying
\[
\mathcal{C}(T,u_0) \leq 2 \sqrt{ a_0 b_0 } \euler^{c_{\ast} / T} \lVert u_0 \rVert_{\cL^2(\Lambda)} ,
\]
where
\begin{align*}
 a_0 &= (\delta / G)^{- K ( 1 + G^{4/3} \lVert V \rVert_\infty^{2/3})},\\
 b_0 &= \euler^{2 \lVert V \rVert_\infty},  \\
 c_{\ast} &\leq \ln (G/\delta)^2  \left( KG + 4 / \ln 2 \right)^2\ \text{and}\\
 K &= K (d)\ \text{is the constant from Theorem \ref{thm:NakicTTV}}.
 \end{align*}
\end{theorem}
\begin{remark}
The same result holds also in the case of controlled heat equation with periodic or Neumann boundary conditions with obvious modifications. 	
\end{remark}
%
\subsection*{Acknowledgements}
The last named author would like to thank the organizers of
the School on Random Schr\"odinger Operators
and the International Conference on Spectral Theory and Mathematical Physics for the invitation and the hospitality
at the Pontificia Universidad Catolica de Chile,
Tomas Lungenstrass for taking notes of the minicourse,
and J.-M.~Barbaroux, N.~Peyerimhoff, G.~Raikov, C.~Rojas-Molina, A.~R\"uland, and C.~Shirley  for stimulating discussions.
Moreover, the authors thank I.~Naki\'c for ongoing discussions on control theory for the heat equation, T.~Kalmes for a careful reading of this manuscript.

\newcommand{\etalchar}[1]{$^{#1}$}
\providecommand{\bysame}{\leavevmode\hbox to3em{\hrulefill}\thinspace}
\providecommand{\MR}{\relax\ifhmode\unskip\space\fi MR }
\providecommand{\MRhref}[2]{%
  \href{http://www.ams.org/mathscinet-getitem?mr=#1}{#2}
}
\providecommand{\href}[2]{#2}

%

\begin{thebibliography}{HKN{\etalchar{+}}06}

\bibitem[BK05]{BourgainK-05}
J.~Bourgain and C.~E. Kenig, \emph{On localization in the continuous
  {A}nderson-{B}ernoulli model in higher dimension}, Invent. Math. \textbf{161}
  (2005), no.~2, 389--426.

\bibitem[BK13]{BourgainK-13}
J.~Bourgain and A.~Klein, \emph{Bounds on the density of states for
  {S}chr{\"o}dinger operators}, Invent. Math. \textbf{194} (2013), no.~1,
  41--72.

\bibitem[BML15]{BrooksML-15}
S.~Brooks, E.~Le Masson, and E.~Lindenstrauss, \emph{Quantum ergodicity and
  averaging operators on the sphere}, arXiv:1505.03887, 2015.

\bibitem[BSS88]{ButzerSS-88}
{P.~L.} Butzer, W.~Splittst\"o\ss{}er, and {R.~L.} Stens, \emph{The sampling
  theorem and linear prediction in signal analysis}, Jber. d. Dt. Math.-Verein.
  \textbf{90} (1988), no.~1, 1--70.

\bibitem[BSSP03]{BoechererSS-03}
S.~B\"ocherer, P.~Sarnak, and R.~Schulze-Pillot, \emph{Arithmetic and
  equidistribution of measures on the sphere}, Commun. Math. Phys. \textbf{242}
  (2003), no.~1-2, 67--80.

\bibitem[Car39]{Carleman-39}
T.~Carleman, \emph{Sur un probl\'eme d'unicit\'e pour les syst\`emes
  d'\'equations aux d\'eriv\'ees partielles \`a deux variables
  ind\'ependantes}, Ark. Mat. Astron. Fysik \textbf{26B} (1939), no.~17, 1--9.

\bibitem[CH94]{CombesH-94}
J.~M. Combes and P.~D. Hislop, \emph{Localization for some continuous random
  hamiltonions in d-dimensions}, J. Funct. Anal. \textbf{124} (1994), no.~1,
  149--180.

\bibitem[CHK03]{CombesHK-03}
J.-M. Combes, P.~D. Hislop, and F.~Klopp, \emph{H\"older continuity of the
  integrated density of states for some random operators at all energies}, Int.
  Math. Res. Not. (2003), no.~4, 179--209. \MR{1 935 272}

\bibitem[CHM96]{CombesHM-96}
J.~M. Combes, P.~D. Hislop, and E.~Mourre, \emph{Spectral averaging,
  perturbation of singular spectra, and localization}, Trans. Amer. Math. Soc.
  \textbf{348} (1996), no.~12, 4883--4894.

\bibitem[CHN01]{CombesHN-01}
J.~M. Combes, P.~D. Hislop, and S.~Nakamura, \emph{The $l^p$-theory of the
  spectral shift function, the {W}egner estimate, and the integrated density of
  states for some random operators}, Commun. Math. Phys. \textbf{218} (2001),
  no.~1, 113--130.

\bibitem[CRT06]{CandesRT-06}
E.~J. Cand{\'e}s, J.~Romberg, and T.~Tao, \emph{Stable signal recovery from
  incomplete and inaccurate measurements}, Comm. Pure Appl. Math. \textbf{59}
  (2006), no.~8, 1207--1223.

\bibitem[DF88]{DonnellyF-88}
H.~Donnelly and C.~Fefferman, \emph{Nodal sets for eigenfunctions on
  {R}iemannian manifolds}, Invent. math. \textbf{93} (1988), no.~1, 161--183.

\bibitem[dV85]{ColindeVerdiere-85}
Y.~Colin de~Verdiere, \emph{Ergodicit\'e et fonctions propres du laplacien},
  Commun. Math. Phys. \textbf{102} (1985), no.~3, 497--502.

\bibitem[EV03]{EscauriazaV-03}
L.~Escauriaza and S.~Vessella, \emph{Optimal three cylinder inequalities for
  solutions to parabolic equations with {L}ipschitz leading coefficients},
  Inverse Problems: Theory and Applications (G.~Alessandrini and G.~Uhlmann,
  eds.), Contemp. Math., vol. 333, American Mathematical Society, 2003,
  pp.~79--87.

\bibitem[Eva98]{Evans-98}
L.~C. Evans, \emph{Partial differential equations}, Graduate Studies in
  Mathematics, vol.~19, American Mathematical Society, Providence, 1998.

\bibitem[FI96]{Fursikov-96}
A.~V. Fursikov and O~Y. Imanuvilov, \emph{Controllability of evolution
  equations}, Lecture Notes Series, no.~34, Seoul National University, 1996.

\bibitem[FR13]{FoucartR-13}
S.~Foucart and H.~Rauhut, \emph{A mathematical introduction to compressive
  sensing}, Birkh\"auser, Basel, 2013.

\bibitem[GK13]{GerminetK-13b}
F.~Germinet and A.~Klein, \emph{A comprehensive proof of localization for
  continuous {A}nderson models with singular random potentials}, J. Eur. Math.
  Soc. \textbf{15} (2013), no.~1, 53--143.

\bibitem[Had03]{Hadamard-03}
J.~Hadamard, \emph{Le{\c{c}}ons sur la propagation des ondes et les
  {\'e}quations de l{'}hy\-dro\-dynamique}, A. Hermann, Paris, 1903.

\bibitem[HKN{\etalchar{+}}06]{HundertmarkKNSV-06}
D.~Hundertmark, R.~Killip, S.~Nakamura, P.~Stollmann, and I.~Veseli\'c,
  \emph{Bounds on the spectral shift function and the density of states},
  Commun. Math. Phys. \textbf{262} (2006), no.~2, 489--503.

\bibitem[H{\"o}r69]{Hoermander-69}
L.~H{\"o}rmander, \emph{Linear partial differential operators}, Springer,
  Berlin, 1969.

\bibitem[HV06]{HelmV-06-arxiv}
M.~Helm and I.~Veseli{\'c}, \emph{A linear {W}egner estimate for alloy type
  {S}chr\"odinger operators on metric graphs}, arXiv:math/0611609, 2006.

\bibitem[HV07]{HelmV-07}
\bysame, \emph{A linear {W}egner estimate for alloy type {S}chr\"odinger
  operators on metric graphs}, J. Math. Phys. \textbf{48} (2007), no.~9,
  092107.

\bibitem[JK85]{JerisonK-85}
D.~Jerison and {C.~E.} Kenig, \emph{Unique continuation and absence of positive
  eigenvalues for {S}chr\"odinger operators}, Ann. Math. \textbf{121} (1985),
  no.~3, 463--494.

\bibitem[JL99]{JerisonL-99}
D.~Jerison and G.~Lebeau, \emph{Nodal sets of sums of eigenfunctions}, Harmonic
  analysis and partial differential equations (M.~Christ, {C.~E.} Kenig, and
  C.~Sadosky, eds.), Chicago Lecture notes in Mathematics, The University of
  Chicago Press, Chicago, 1999, pp.~223--239.

\bibitem[Ken86]{Kenig-86}
{C.~E.} Kenig, \emph{Carleman estimates, uniform sobolev inequalities for
  second-order differential operators, and unique continuation theorems},
  Proceedings of the International Congress of Mathematicians, 1986 ({A.~M.}
  Gleason, ed.), American Mathematical Society, 1986, pp.~948--960.

\bibitem[Kir96]{Kirsch-96}
W.~Kirsch, \emph{{Wegner} estimates and {Anderson} localization for alloy-type
  potentials}, Math. Z. \textbf{221} (1996), no.~1, 507--512.

\bibitem[Kle13]{Klein-13}
A.~Klein, \emph{Unique continuation principle for spectral projections of
  {S}chr{\"o}dinger operators and optimal {W}egner estimates for non-ergodic
  random {S}chr{\"o}dinger operators}, Commun. Math. Phys. \textbf{323} (2013),
  no.~3, 1229--1246.

\bibitem[Kov01]{Kovrijkine-01}
O.~Kovrijkine, \emph{Some results related to the {L}ogvinenko-{S}ereda
  theorem}, Proc. Amer. Math. Soc. \textbf{129} (2001), no.~10, 3037--3047.

\bibitem[KRS86]{KenigRS-86}
{C.~E.} Kenig, A.~Ruiz, and C.~Sogge, \emph{Remarks on unique continuation
  theorems}, Seminarios, U.A.M. Madrid, 1986.

\bibitem[Kuk98]{Kukavica-98}
I.~Kukavica, \emph{Quantitative uniqueness for second-order elliptic
  operators}, Duke Math. J. \textbf{91} (1998), no.~2, 225--240.

\bibitem[KV02]{KirschV-02}
W.~Kirsch and I.~Veseli{\'c}, \emph{Existence of the density of states for
  one-dimensional alloy-type potentials with small support}, Mathematical
  Results in Quantum Mechanics (R.~Weber, P.~Exner, and B.~Gr{\'e}bert, eds.),
  Contemp. Math., vol. 307, American Mathematical Society, 2002, pp.~171--176.

\bibitem[KV10]{KirschV-10}
\bysame, \emph{Lifshitz tails for a class of {S}chr\"odinger operators with
  random breather-type potential}, Lett. Math. Phys. \textbf{94} (2010), no.~1,
  27--39.

\bibitem[Lit12]{Littlewood-1912}
J.~E. Littlewood, \emph{Contr\^o{}le exact de l'\'e{}quation de la chaleur}, C.
  R. Acad. Sci. \textbf{154} (1912), 335--356.

\bibitem[LR95]{LebeauR-95}
G.~Lebeau and L.~Robbiano, \emph{Contr\^o{}le exact de l'\'e{}quation de la
  chaleur}, Commun. Part. Diff. Eq. \textbf{20} (1995), no.~1--2, 335--356.

\bibitem[LS74]{LogvinenkoS-74}
{V.~N.} Logvinenko and {Ju.~F.} Sereda, \emph{Equivalent norms in spaces of
  entire functions of exponential type}, Teor. Funkci\u\i\ Funkcional. Anal. i
  Prilo\v zen \textbf{175} (1974), no.~20, 102--111.

\bibitem[Mes92]{Meshkov-92}
{V.~Z.} Meshkov, \emph{On the possible rate of decay at infinity of solutions
  of second order partial differential equations}, Math. USSR Sb. \textbf{72}
  (1992), no.~2, 343--361.

\bibitem[Mil10]{Miller-10}
L.~Miller, \emph{A direct {L}ebeau-{R}obbiano strategy for the observability of
  heat-like semigroups}, Discrete Cont. Dyn.-B \textbf{14} (2010), no.~4,
  1465--1485.

\bibitem[MS13]{MuscaluS-13}
C.~Muscalu and W.~Schlag, \emph{Classical and multilinear harmonic analysis},
  Cambridge Studies in Advanced Mathematics, vol. 137, Cambridge University
  Press, Cambridge, 2013.

\bibitem[M{\"u}l54]{Mueller-54}
C.~M{\"u}ller, \emph{On the behaviour of the solution of the differential
  equation $\delta u = f (x,u)$ in the neighborhood of a point}, Commun Pur.
  Appl. Math. \textbf{7} (1954), no.~3, 505--515.

\bibitem[NTTV]{NakicTTV-prep}
I.~Naki\'c, M.~T\"aufer, M.~Tautenhahn, and I.~Veseli\'c, In preparation.

\bibitem[NTTV15]{NakicTTV-14}
\bysame, \emph{Scale-free uncertainty principles and {W}egner estimates for
  random breather potentials}, C. R. Math. (2015), Doi:
  10.1016/j.crma.2015.08.005.

\bibitem[OCP13]{Ortega-CerdaP-13}
J.~Ortega-Cerd\`a{} and B.~Pridhnani, \emph{Carleson measures and
  {L}ogvinenko--{S}ereda sets on compact manifolds}, Forum Math. \textbf{25}
  (2013), no.~1, 151--172.

\bibitem[RMV13]{Rojas-MolinaV-13}
C.~Rojas-Molina and I.~Veseli{\'c}, \emph{Scale-free unique continuation
  estimates and applications to random {S}chr{\"o}dinger operators}, Commun.
  Math. Phys. \textbf{320} (2013), no.~1, 245--274.

\bibitem[RT15]{RodnianskiT-15}
I.~Rodnianski and T.~Tao, \emph{Effective limiting absorption principles, and
  applications}, Commun. Math. Phys. \textbf{333} (2015), no.~1, 1--95.

\bibitem[Rud70]{Rudin-70}
W.~Rudin, \emph{Real and complex analysis}, McGraw Hill, London, 1970.

\bibitem[Shi14]{Shirley-14-dissertation}
C.~Shirley, \emph{Statistiques spectrales d'op\'erateurs de {Schr\"odinger}
  al\'eatoires unidimensionnels}, Ph.D. thesis, Universit\'e {Pierre} et
  {Marie} {Curie}, Paris, 2014.


\bibitem[Shi15]{Shirley-15}
C.~Shirley, \emph{Decorrelation estimates for some continuous and discrete random Schr{\"o}dinger operators in dimension one, without covering condition},
arXiv:1505.06112 [math-ph], 2015.

\bibitem[Sto01]{Stollmann-01}
P.~Stollmann, \emph{Caught by disorder: Bound states in random media}, Progress
  in Mathematical Physics, vol.~20, Birkh{\"a}user, Boston, 2001.

\bibitem[TV15a]{TaeuferV}
M.~T\"aufer and I.~Veseli{\'c}, \emph{Conditional {W}egner estimate for the
  standard random breather potential}, J. Stat. Phys. \textbf{161} (2015),
  no.~4, 902--914.

\bibitem[TV15b]{TautenhahnV-15}
M.~Tautenhahn and I.~Veseli{\'c}, \emph{Sampling inequality for {$L^2$}-norms
  of eigenfunctions, spectral projectors, and {W}eyl sequences of
  {S}chr\"o{}dinger operators}, arXiv:1504.00554 [math.AP], 2015.

\bibitem[Ves96]{Veselic-96}
I.~Veseli\'{c}, \emph{Lokalisierung bei zuf\"allig gest\"orten periodischen
  {Schr\"odinger\-opera\-toren} in {Dimension Eins}}, Diplomarbeit,
  Ruhr-Universit\"at Bochum, 1996, available at
  http://www.ruhr-uni-bochum.de/mathphys/ivan/diplomski-www-abstract.htm.

\bibitem[Ves08]{Veselic-08}
I.~Veseli{\'c}, \emph{Existence and regularity properties of the integrated
  density of states of random {S}chr{\"o}dinger operators}, Lecture Notes in
  Mathematics, vol. 1917, Springer, 2008.

\bibitem[Wol93]{Wolff-93}
T.~H. Wolff, \emph{Recent work on sharp estimates in second-order elliptic
  unique continuation problems}, J. Geom. Anal. \textbf{3} (1993), no.~6,
  621--650.

\bibitem[Zel92]{Zelditch-92}
S.~Zelditch, \emph{Quantum ergodicity on the sphere}, Comm. Math. Phys.
  \textbf{146} (1992), no.~1, 61--71.

\end{thebibliography}
%
\end{document}